\documentclass[12pt]{article}
\usepackage[british]{babel}
\usepackage[babel]{csquotes}
\usepackage{amsmath}
\usepackage{amssymb}
\usepackage{amsthm}     
\usepackage{amsfonts}   
\usepackage{mathtools}  
\usepackage{mathrsfs}   
\usepackage{bbm}        
\usepackage{ragged2e}

\usepackage{graphicx}
\usepackage{caption}
\usepackage{subcaption}
\usepackage{float}

\usepackage{tikz}
\usepackage{pgfplots} 
\usepackage{lscape} 

\usepackage{hyperref}
\usepackage{enumitem}
\usepackage{ulem}
\usepackage{url}
\usepackage{eurosym}

\newtheorem{defi}{Definition}[section]
\newtheorem{ass}[defi]{Assumption}

\newtheorem{thm}[defi]{Theorem}
\newtheorem{lem}[defi]{Lemma}
\newtheorem{prop}[defi]{Proposition}

\newtheorem{cor}[defi]{Corollary}
\newtheorem{rem}[defi]{Remark}

\def\P{\mathbb{P}}								
\def\E{\mathbb{E}}                              

\newcommand\expec[1]{\mathbb{E}\left[#1\right]}	    
\newcommand\prob[1]{\mathbb{P}\left(#1\right)}	

\newcommand{\sph}{\mathbb{S}}

\newcommand\R{\mathbb{R}}
\newcommand\N{\mathbb{N}}
\newcommand\Z{\mathbb{Z}}

\newcommand\eps{\varepsilon}
\newcommand\vphi{\varphi}

\newcommand\abs[1]{\left\lvert #1 \right\rvert}  
\newcommand\norm[1]{\left\Vert #1 \right\Vert}
\newcommand\abrac[1]{\left\langle #1 \right\rangle}            

\DeclareMathOperator\spann{span}
\DeclareMathOperator\sign{sign}

\let\BFseries\bfseries\def\bfseries{\BFseries\mathversion{bold}}
\pgfplotsset{compat=1.18}

\RequirePackage{amsthm,amsmath,amsfonts,mathrsfs,amssymb,hyperref}

\DeclareMathAlphabet{\mathdutchcal}{U}{dutchcal}{m}{n}
\SetMathAlphabet{\mathdutchcal}{bold}{U}{dutchcal}{b}{n}
\DeclareMathAlphabet{\mathdutchbcal}{U}{dutchcal}{b}{n}
\newcommand{\ind}{1\hspace{-0.098cm}\mathrm{l}}
\newcommand{\indi}[1]{\,\ind_{\{#1\}}}
\def\dd{\mbox{d}}
\def\E{\mathbb{E}}

\def\P{\mathbb{P}}
\def\R{\mathbb{R}}
\def\eps{\varepsilon}
\def\sign{\operatorname*{{sign}}}

\def\deq{\mathrel{\stackrel{d}{=}}} 
\def\S{\mathbb{S}}

\usepackage{enumitem}
\setenumerate[1]{label=(\textup{\alph*})}
\setenumerate[2]{label=(\textup{\roman*})}

\newcommand{\ph}{\frac{\pi}{2}}
\newcommand{\eh}{\frac{1}{2}}
\newcommand{\hah}{(H+\alpha)/2}
\newcommand{\ltw}{L_2^w[-1,1]}
\newcommand{\bb}{\mathcal{B}}
\newcommand{\hh}{\mathcal{H}}

\newcommand{\hb}{\mathbb{H}}

\newcommand{\OO}{\mathcal{O}}
\newcommand{\ol}[1]{\overline{#1}}
\newcommand{\vol}[1]{\abs{\sph_{#1}}}

\newcommand{\an}[1]{a_n^{(#1)}}
\newcommand{\anbr}[1]{\left[a_n^{\left(#1\right)}\right]}
\newcommand{\bn}[1]{b_n^{(#1)}}
\newcommand{\bnbr}[1]{\left[b_n^{\left(#1\right)}\right]}
\newcommand{\fda}{f_{\delta,\alpha}}
\newcommand{\gh}{G^{(H)}}

\title{Persistence probabilities of spherical fractional Brownian motion}
\author{Frank Aurzada and  Max Helmer}

\begin{document}

\maketitle

\begin{abstract}
We compute the rate of decay of the persistence probabilities of spherical fractional Brownian motion, which was defined by L\'evy \cite{Levy1965} and Istas \cite{sphereFBM}. The rate resembles the Euclidean case treated in \cite{molchan99}.

As a by-product we consider the coefficients of series representations of functions with algebraic endpoint singularities in terms of re-scaled Gegenbauer polynomials, which partly generalises \cite{Sidi09}.
\end{abstract}



\allowdisplaybreaks

\section{Introduction}

\subsection{Main result}
The subject of this paper is situated in the area of stochastic processes under constraints. An important theme here is the study of persistence probabilities, which are meant as a tool to study the behaviour of stochastic processes when they have large sections away from a boundary, cf.\ the monograph \cite{metzler} and the surveys \cite{bray} and \cite{aurzadasimon} on the subject. In particular, one is interested in understanding the repulsion effect of the boundary. In this work, we contribute to this by studying Gaussian random fields (i.e.\ processes with multi-dimensional index sets) and ask what is the effect of the geometry of the index set on the persistence probabilities. We refer to the survey \cite{lodhiaetal2016} with extensive bibliography
as a reference on fractional Gaussian fields. In this work, we consider fractional Brownian motion (FBM) indexed by the sphere \cite{sphereFBM}. 

Let us start by recalling the definition of usual FBM first: A stochastic process $(B_H(t))_{t\in\R}$ (indexed by the real line) is called fractional Brownian motion if it is a continuous, centred Gaussian process with covariance function
\begin{equation} \label{eqn:fbmcov}
\E[ B_H(t) B_H(s) ] = \frac{1}{2} \left( |t|^{2H}+|s|^{2H} - |t-s|^{2H} \right), \qquad t,s\in \R.
\end{equation}
Here, the number $H\in(0,1)$ is a parameter (called the Hurst parameter). 
FBM itself is a generalisation of Brownian motion, which is included as the special case $H=1/2$. We refer to the monographs \cite{biaginibook,mishurabook,nourdinbook} for a detailed account.
 
Let us now consider the case of $\R^d$ as index set. For any $H\in(0,1)$ there is a continuous, centred Gaussian process $(B_H(t))_{t\in\R^d}$ (indexed by $\R^d$) with covariance function
\begin{equation} \label{eqn:fbmcovd}
\E[ B_H(t) B_H(s) ] = \frac{1}{2} \left( ||t||_2^{2H}+||s||_2^{2H} - ||t-s||_2^{2H} \right), \qquad t,s\in \R^d,
\end{equation}
where $||.||_2$ denotes the usual Euclidean norm on $\R^d$ and we stress the analogy to (\ref{eqn:fbmcov}). 

In this context, Molchan showed the following result concerning persistence probabilities \cite{molchan99} (also see (\ref{eqn:molchanformulation2}) below): For FBM $(B_{H}(t))_{t\in\R^d}$ one has
        \begin{align}
            \P\left( \sup_{t\in[-1,1]^d} B_H(t) < \eps\right) = \eps^{\frac{d}{H} + o(1)}, \qquad \text{ as } \eps\to 0. \label{eqn:molchanddim}
        \end{align}

Our main result deals with the sphere instead of the Euclidean space. We fix a dimension $d\geq 2$ for the rest of the paper and consider the index set
\begin{align*}
\S_{d-1} := \{ \eta \in \R^d : ||\eta||_2 = 1 \}.
\end{align*}
Further, we let $d(\eta,\zeta):=\arccos(\langle \eta,\zeta\rangle)$ denote the geodesic distance of two points $\eta,\zeta\in\S_{d-1}$ on the sphere. Let $O\in \S_{d-1}$ be an arbitrary point that we fix for the rest of this paper (`the north pole'). We recall from \cite{sphereFBM} that for any $0<H\leq 1/2$ there exists a continuous, centred Gaussian process $(S_H(\eta))_{\eta\in\S_{d-1}}$ with
\begin{equation}
\E[ S_H(\eta) S_H(\zeta) ] = \frac{1}{2} \left( d(\eta,O)^{2H}+d(\zeta,O)^{2H} - d(\eta,\zeta)^{2H} \right), \qquad \eta,\zeta\in \S_{d-1}, \label{eqn:covspherical}
\end{equation}
in analogy to (\ref{eqn:fbmcov}) and (\ref{eqn:fbmcovd}). Note that the definition yields $S_H(O)=0$ almost surely.

A process $(S_H(\eta))_{\eta\in\S_{d-1}}$ with (\ref{eqn:covspherical}) is called spherical fractional Brownian motion (SFBM) with Hurst parameter $H$ in dimension $d$. For $H=1/2$, the process was studied already by L\'evy \cite{Levy1965} and is also referred to as L\'evy's spherical Brownian motion. We stress that for $H>1/2$, the right-hand side of (\ref{eqn:covspherical}) is not a covariance function so that no SFBM exists for $H>1/2$, cf.\ \cite{sphereFBM}.

Given this notation, we can state our main result.

\begin{thm} \label{thm:main}
       For SFBM $(S_{H}(\eta))_{\eta\in\S_{d-1}}$ we have
        \begin{align} \label{eqn:main}
            \P\left( \sup_{\eta\in \S_{d-1}} S_{H}(\eta) < \eps\right) = \eps^{\frac{d-1}{H} + o(1)}, \qquad \text{ as } \eps\to 0.
        \end{align}
\end{thm}

This result should be compared to the Euclidean case cited in (\ref{eqn:molchanddim}). There, a $d$-dimensional index set is considered, while $\S_{d-1}$ is a $(d-1)$-dimensional manifold. In fact, in our proofs, we shall compare local `pieces' of the persistence probability in (\ref{eqn:main}) to $(d-1)$-dimensional Euclidean persistence probabilities that can be handled with a variant of (\ref{eqn:molchanddim}). In order to make this kind of comparison, we employ Toponogov's theorem from Riemannian geometry. Another ingredient of the proof is a comparison to the `area' that SFBM spends positive $Z_+:=\int_{\S_{d-1}} \ind_{S_H(\xi)>0} \dd \sigma(\xi)$ (where $\sigma$ is the surface measure on $\S_{d-1}$), the distribution of which was found recently in \cite{aurdorpit2024}.
Further, we employ a change of measure technique, which in turn requires precise knowledge about the reproducing kernel Hilbert space (RKHS) of SFBM.

The general strategy of approaching the RHKS of SFBM through generalised Legendre polynomials is inspired by \cite{Istas2006} and \cite[Chapter 4.2]{Istas12Manifold}. Contrary to the mentioned papers, we will need very precise information about the asymptotics of the coefficients of the Legendre series of concrete functions. Obtaining such asymptotic behaviour requires a generalisation of a result in \cite{Sidi09} concerning the asymptotics of series coefficients for functions with end point singularities. Our result in this direction (cf.\ Theorem~\ref{thm_LegendreCoeffOrder} below) may be of independent interest.

The outline of the paper is as follows. We continue in Section~\ref{sec:related} with a discussion of related works and formulate a number of open questions attached to those and to our main result. The goal of Section~\ref{sec_RKHS} is to give a representation of the reproducing kernel Hilbert space of SFBM. Here, generalised Legendre polynomials and series representations w.r.t.\ the latter will play a crucial role.
In Section~\ref{sec_polynomialApprox}, we spell out the mentioned generalisation of the result in \cite{Sidi09} concerning the asymptotics of the Legendre series coefficients of functions with end point singularities. Using this preliminary work, we can then define a suitable concrete function in the RKHS of SFBM in Section~\ref{sec_existencePartFktRKHS}. Finally, Section~\ref{sec_proofMain} is dedicated to the proof of Theorem~\ref{thm:main}.

\subsection{Related work and open problems} \label{sec:related}
Let us review some related results and list a number of open questions in this context. 

\paragraph*{The Euclidean case.} First of all, let us mention that it is more common to cite the result (\ref{eqn:molchanddim}) in the following form: For an FBM $(B_H(t))_{t\in \R^d}$ on $\R^d$,
\begin{equation}
    \label{eqn:molchanformulation2}
\P\left( \sup_{t\in[-T,T]^d} B_H(t) < 1\right)=T^{-d+o(1)}, \qquad \text{ as } T\to\infty,
\end{equation}
which was proved in \cite[Theorem 3]{molchan99}. We chose the formulation in (\ref{eqn:molchanddim}) in order to allow for a comparison to the case of the sphere, which is a fixed and compact index space. The equivalence of  (\ref{eqn:molchanddim}) and (\ref{eqn:molchanformulation2}) follows immediately from the self-similarity property $(B_H(ct))_{t\in\R^d}\deq(|c|^H B_H(t))_{t\in\R^d}$, for any $c\in\R$. The result (\ref{eqn:molchanformulation2}) can be compared to several related pieces of work.

First of all, in the one-dimensional case it boils down to
\begin{equation}
    \label{eqn:molchanformulation2d1}
\P\left( \sup_{t\in[-T,T]} B_H(t) < 1\right)=T^{-1+o(1)}, \qquad \text{ as } T\to\infty.
\end{equation}
This is in contrast to the `one-sided' persistence result obtained in  \cite[Theorem 1]{molchan99}: For FBM $(B_H(t))_{t \geq 0}$ (indexed by the non-negative real line), one has
\begin{align}
            \P\left( \sup_{t\in [0,T]} B_H(t) < 1\right) = T^{-(1-H)+o(1)}, \qquad \text{ as } T\to\infty. \label{eqn:onesidedmolchan}
\end{align}

We stress that the position of $0$ is crucial. The result (\ref{eqn:onesidedmolchan}) handles the case where $0$ is at the boundary of the index set. On the contrary, (\ref{eqn:molchanddim}) (i.e.\ (\ref{eqn:molchanformulation2})) as well as our main result, Theorem~\ref{thm:main}, deal with the case when $0$ (or the north pole $O$, respectively) is within the index set.

It would be of great interest to shed more light on the role of the position of $0$ at the boundary. One work on this direction is \cite{molchan2018} (also see \cite{molchan2017}): If the fixed bounded domain $K\subseteq \R^d$ has a smooth boundary at $0$ (i.e.\ $0\in\partial K$ and there is a ball containing $0$ that is contained in $K$: i.e.\ $\exists r>0, t_0\in\R^d : 0\in \{ t : ||t-t_0||\leq r\}\subseteq K$) then
        \begin{align} \label{eqn:zeroinaballmolchan}
            \P\left(\sup_{t\in T K} B_H(t) < 1\right) = T^{-(d-H) + o(1)}, \qquad \text{ as } T\to\infty;
        \end{align}
and one might write the exponent as $(d-1)+(1-H)$ and compare to (\ref{eqn:molchanformulation2}) and (\ref{eqn:onesidedmolchan}). In particular, this is true for the example $K=[0,1]\times[-1,1]^{d-1}$.

These results lead Molchan to conjecture (see e.g.\ \cite{molchan2017}) that for the domains $K := [0,1]^k \times [-1,1]^{d-k}$ (where $0\leq k\leq d$) it should be true that
\begin{align*}
            \P\left( \sup_{t\in T K} B_H(t) < 1\right) = T^{-(d-k H)  + o(1)}, \qquad \text{ as } T\to\infty,
\end{align*}
and again the exponent should be read as $(d-k)+k\cdot(1-H)$, cf.\ (\ref{eqn:molchanformulation2}) and (\ref{eqn:onesidedmolchan}). This conjecture is true for $k=0$ (which is (\ref{eqn:molchanformulation2})) and $k=1$ (which follows from (\ref{eqn:zeroinaballmolchan})). In particular, it would be interesting to find the asymptotics in the case $k=d$, which would yield a $d$-dimensional generalisation of (\ref{eqn:onesidedmolchan}).

Yet another result is given in \cite{molchan2012} and considers an asymmetric one-dimensional domain:
        \begin{align*}
            \P\left( \sup_{t\in [-T^\alpha,T]} B_H(t) < 1 \right) = T^{ -[(1-\alpha)(1-H)\; + \; \alpha \cdot 1] + o(1)}, \qquad \text{ as } T\to\infty,
        \end{align*}
where w.l.o.g.\ $\alpha \in [0,1]$, and we stress that the persistence exponent is precisely the convex combination of the exponents from (\ref{eqn:molchanformulation2d1}) and (\ref{eqn:onesidedmolchan}). A $d$-dimensional analogue does not seem to be available.

\paragraph*{The spherical case.} 
Similar questions as in the Euclidean case can also be phrased in the setup of spherical fractional Brownian motion. To formulate one corresponding question, fix an additional arbitrary point ${\bf 1}\in\S_{d-1}$ with $\langle {\bf 1},O\rangle=0$ (i.e.\ ${\bf 1}$ lies `on the equator' of $\S_{d-1}$). For points $\eta\in\S_{d-1}\setminus \{ O,\ol{O}\}$,
where $\ol{O}$ is the point antipodal to $O$,
we consider the projection of $\eta$ to the equator: $\mathdutchcal{e}(\eta):=\frac{\eta-\langle\eta,O\rangle O}{||\eta-\langle\eta,O\rangle O||_2}$ (and set $\mathdutchcal{e}(O):=\mathdutchcal{e}(\ol{O}):={\bf 1}$). Then for an angle $\psi\in[0,2\pi]$, consider the set of points on the sphere that have a `longitude difference' to the point $1$ of at most $\psi/2$:
\begin{align*}
L_\psi := \{ \eta \in \S_{d-1} : d(\mathdutchcal{e}(\eta),{\bf 1}) \leq \psi/2 \}.
\end{align*}
It is reasonable to ask whether there exists and what is the value of $\theta=\theta(\psi,d,H)$ such that
\begin{align*}
\P\left( \sup_{\eta\in L_\psi} S_H(\eta) < \eps\right) = \eps^{ \frac{\theta}{H} + o(1)}, \qquad \text{ as } \eps\to 0.
\end{align*}
Our main result, Thereom~\ref{thm:main} states that $\theta(2\pi,d,H)=d-1$. A reduction argument to the one-dimensional case shows $\theta(0,d,H)=1-H$. Even in the case $d=3$, it would be very interesting to find $\theta(\psi,d,H)$ for $0<\psi<2\pi$. Further, given the result in (\ref{eqn:zeroinaballmolchan}) one is lead to conjecture that $\theta(\psi,d,H)=d-1-H$ for $\pi\leq \psi<2\pi$.

\paragraph*{The hyperbolic case.} Changing the geometry in a different way, one can also consider the real hyperbolic space $\mathbb{H}_{d-1}$:
        \begin{align*}
            \mathbb{H}_{d-1} := \left\{ (x_1,\ldots,x_{d}) \in\R^{d} \, :\,  x_1 > 0, \quad x_1^2 - \sum_{i=2}^{d} x_i^2 = 1 \right\}.
        \end{align*}
Existence of a hyperbolic fractional Brownian motion for $0<H\leq 1/2$ was also established by Istas \cite{sphereFBM}. Here, no result for persistence probabilities is available and similarly to the Euclidean and to the spherical case, many open questions can be formulated.

\section{The reproducing kernel Hilbert space of SFBM}\label{sec_RKHS}
\subsection{Preliminaries on orthogonal polynomials}\label{sec_prelimPoly}

In this section, we introduce Gegenbauer polynomials and, as a special case, generalised Legendre polynomials and collect some facts about these families of orthogonal polynomials from the literature that we need in the sequel.

Let $(C_n^{(\lambda)})_{n\in\N_0}$ be the Gegenbauer polynomials (cf.\ \cite[Introduction]{wang16}, \cite[Sec.\ 4.2.2]{Istas12Manifold}), also known as ultraspherical polynomials, with parameter $0 \neq \lambda > -\eh$ and $\deg C_n^{(\lambda)} = n$. They are an orthogonal set of polynomials w.r.t.\ the weight function $w(t) := (1-t^2)^{\lambda-\eh}$ on $[-1,1]$, i.e.\
\begin{align*}
    \int_{-1}^1 C_n^{(\lambda)}(t) C_m^{(\lambda)}(t) (1-t^2)^{\lambda-\eh} \dd t = 0,\qquad \text{for $n\neq m$},
\end{align*}
and are uniquely defined by Rodrigues' formula
\begin{align}\label{eq_RodriguesGegenbauerDef}
    C_n^{(\lambda)}(t) := \frac{(-1)^n}{2^n n!} \frac{\Gamma(\lambda + \eh) \Gamma(n+2\lambda)}{\Gamma(2\lambda)\Gamma(n+\lambda+\eh)} (1-t^2)^{\eh-\lambda} \frac{\partial^n}{\partial t^n} \left[ (1-t^2)^{\lambda-\eh} (1-t^2)^n \right],
\end{align}
where $\Gamma$ is the gamma function. The Gegenbauer polynomials satisfy
\begin{align} \label{eq_normalf}
    C_n^{(\lambda)}(1) = \frac{\Gamma(n+2\lambda)}{\Gamma(2\lambda) n!}.
\end{align}
For our purposes, we require the following re-normalised Gegenbauer polynomials parametrised by a fixed integer $d\geq 2$:
\begin{align*}
    P_n(t) := P_n^{(d)}(t) :=  \frac{C_n^{(\lambda)}(t)
    }{\frac{\Gamma(n+2\lambda)}{\Gamma(2\lambda) n!}},\qquad \text{ with } \lambda = \frac{d}{2}-1. 
\end{align*}
We shall suppress the index $d$ in the sequel for the sake of readability, assuming that the dimension $d\geq 2$ is fixed. Note that the definition also makes sense for $d=2$ since the critical factors $\Gamma(2\lambda)$  and $\Gamma(n+2\lambda)$ (for $n=0$) are formally eliminated from the Rodrigues representation (cf.\ Prop.\ 4.19 in \cite{frye12}).  We call these polynomials the {\it generalised Legendre polynomials} or {\it Legendre polynomials of dimension $d$}. They satisfy
    $P_n(1) = 1$,
for any $d \geq 2$, by definition and (\ref{eq_normalf}). We have chosen this naming convention, since for $d=3$ the generalised Legendre polynomials are exactly the Legendre polynomials. Our main source of reference on this topic is \cite{frye12}, in which a similar naming convention is adopted.

The generalised Legendre polynomials $(P_n)$ have a few more noteworthy properties that we use later on. First of all, they are symmetric for all even $n$ and anti-symmetric for all odd $n$. They are also uniformly bounded, i.e.\ $\abs{P_n(t)} \leq 1$ for all $t\in[-1,1]$ (cf.\ \cite[Proposition 4.15]{frye12}). These polynomials are also complete w.r.t.\ the inner product space $L_2^w[-1,1]$ with inner product defined by
\begin{align*}
    \abrac{f,g} = \int_{-1}^1 f(t)g(t) (1-t^2)^{(d-3)/2} \dd t,\qquad  f,g\in L_2^w[-1,1].
\end{align*}
We refer to this space as the weighted $L_2$-space with weight function $w(t) = (1-t^2)^{(d-3)/2}$. Completeness is implied by the well-known fact that Jacobi polynomials, which are a re-scaled generalisation of the Gegenbauer polynomials, are complete. Due to this completeness, we can introduce a series representation for $L_2^w$-functions w.r.t.\ the generalised Legendre polynomials. Calculating the series coefficients requires knowledge about the norm of the sequence $(P_n)$:
\begin{equation}
     \int_{-1}^1 P_n(t)^2 (1-t^2)^{(d-3)/2} \dd t = \frac{\abs{\sph_{d-1}}}{N(d,n) \abs{\sph_{d-2}}}, \label{eq_LegendreNorm}
\end{equation}  
where
\begin{align*}
 N(d,n):= \begin{cases} \frac{2n + d-2}{n}\binom{n+d-3}{n-1}, & n\geq 1; \\
  1, & n=0,
  \end{cases}
\end{align*}
and we let $\sigma$ be the surface measure on the $\sph_{d-1}$-ball and $\abs{\sph_{d-1}} = \sigma(\sph_{d-1})$ be its total surface area. For reference, see \cite[Theorem 4.4]{frye12}. 

Next, we want to introduce our first asymptotic estimate, which concerns the behaviour of $N(d,n)$ as $n\to\infty$. To express limits we use the standard notation $f\sim g$ if $\lim f/g = 1$.
Throughout the paper, the constant $c > 0$ is used in a generic way and it is made clear what the respective constant depends on.

At a few places in the paper, we compute limits of quotients of Gamma functions, for which we use the well-known asymptotics
\begin{align}\label{eq_gammaQuotient}
        \frac{\Gamma(x+a)}{\Gamma(x+b)} \sim c \; x^{a-b}, \qquad \text{ as } x \to \infty,
\end{align}
for fixed $a,b\in\R$. The constant $c$ depends on $a$ and $b$. We may immediately apply this to obtain that
\begin{align}\label{lem_orderOfNdn}
        N(d,n) \sim  c \; n^{d-2},
\end{align}
as $n\to\infty$, for some constant $c>0$, which only depends on the dimension $d$.

\subsection{Integral representations}\label{sec_integralRep}

In this section, we deal with the following problem: Let $f\in\ltw$. Is there a function $g\in\ltw$ such that  we have the integral representation
\begin{align*}
    f(\abrac{\eta,\zeta}) = \int_{\sph_{d-1}} g(\abrac{\eta, \xi}) g(\abrac{\zeta,\xi}) \dd\sigma(\xi),  \qquad \text{for all } \eta,\zeta\in \sph_{d-1} \; ?
\end{align*}

An answer sufficient for our purposes is given in Proposition~\ref{prop_sqIntegralRep}. We apply this knowledge to give a concrete represenation of the reproducing kernel Hilbert space of SFBM in Section~\ref{sec_repOfRKHSofSFBM}.

Let us start by stating an important connection between spherical integrals and integrals over $[-1,1]$ (cf.\ \cite[Lemma 4.17]{frye12}).

\begin{lem}\label{lem_simpleRadialFunctionIntegration}
    Let $\eta \in\sph_{d-1}$ and let $f\in L_2^w[-1,1]$. Then
    \begin{align*}
        \int_{\sph_{d-1}} f(\abrac{\eta, \xi}) \dd \sigma(\xi)
        = \abs{\sph_{d-2}} \int_{-1}^1 f(t)(1-t^2)^{(d-3)/2} \dd t.
    \end{align*}
\end{lem}

Furthermore, we can integrate w.r.t.\ Legendre polynomials (cf.\ \cite[Lemma 4.23]{frye12}).

\begin{lem}\label{lem_radialFunctionAndLegendreIntegration}
    Let $\eta, \zeta \in\sph_{d-1}$ and let $f\in L_2^w[-1,1]$. Then
    \begin{align*}
        \int_{\sph_{d-1}} f(\abrac{\eta, \xi}) P_n(\abrac{\zeta, \xi}) \dd \sigma(\xi)
        = \abs{\sph_{d-2}} P_n(\abrac{\eta, \zeta}) \int_{-1}^1 f(t) P_n(t) (1-t^2)^{(d-3)/2} \dd t,
    \end{align*}
    where $P_n$ is the $n$-th Legendre polynomial of dimension $d$.
\end{lem}

As stated in the previous subsection, the Legendre polynomials form a complete, orthogonal set w.r.t.\ the weighted inner product space of $L_2^w[-1,1]$-functions. In particular, for any $f\in L_2^w[-1,1]$ there exists a series representation
\begin{align*}
    f(t) = \sum_{n=0}^\infty a_n P_n(t),
\end{align*}
which converges in $L_2^w[-1,1]$. The coefficients are given by
\begin{align}\label{legendre_orthogonal_decomposition}
    a_n = \frac{N(d,n) \vol{d-2} }{\vol{d-1}} \int_{-1}^1 f(t) P_n(t) (1-t^2)^{(d-3)/2} \dd t,
\end{align}
which can be inferred from the $L_2^w[-1,1]$ norm of the Legendre polynomials given in (\ref{eq_LegendreNorm}). We refer to the sequence $(a_n)$ as the Legendre series coefficients of the function $f$. Sometimes we can infer additional convergence properties.
\begin{lem}\label{lem_CoeffSummable}
    If the Legendre coefficients $(a_n)$ of a function $f$ are absolutely summable then the series converges uniformly on $[-1,1]$ and $f$ is necessarily uniformly continuous.
\end{lem}
\begin{proof}
    This is an immediate consequence of the summability assumption and the fact that $\abs{P_n(t)}\leq 1$ is uniformly bounded, which implies that the Legendre series is a normally convergent series. The limit of a normally convergent sequence must be continuous, and continuous functions on compact intervals are uniformly continuous.
\end{proof}

The following lemma lets us obtain the Legendre coefficients of the integral of a product of $\ltw$ functions. Since the proof is obtained by combining Lemma~\ref{lem_radialFunctionAndLegendreIntegration} with the representation of functions in $\ltw$ by their Legendre series using standard arguments, we omit it.

\begin{lem}\label{lem_productIntegration}
Let $f(t) = \sum_{n=0}^\infty a_n P_n(t)$ and $g(t) = \sum_{n=0}^\infty b_n P_n(t)$ be two functions in $L_2^w[-1,1]$ with the given series representation in terms of the basis of Legendre polynomials of dimension $d\geq 2$. Then for any $\eta,\zeta\in\sph_{d-1}$ we have
\begin{align*}
    \int_{\sph_{d-1}} f(\abrac{\eta,\xi})\; g(\abrac{\zeta,\xi}) \dd \sigma(\xi) = \vol{d-1} \sum_{n=0}^\infty \frac{ a_n b_n }{N(d,n)} P_n(\abrac{\eta, \zeta}).
\end{align*}
\end{lem}

The question of whether the Legendre series coefficients of a series representation are all positive and absolutely summable is central to our proof. If they are we can define a a useful integral representation. 

\begin{prop}\label{prop_sqIntegralRep}
    Let $f\in L_2^w[-1,1]$. If the Legendre coefficients of $f$ are non-negative and absolutely summable then there exists a function $g\in L_2^w[-1,1]$ such that
    \begin{align*}
        f(\abrac{\eta,\zeta}) = \int_{\sph_{d-1}} g(\abrac{\eta, \xi}) g(\abrac{\zeta,\xi}) \dd\sigma(\xi).
    \end{align*}
    In this case the Legendre series of $f$ converges in $\ltw$ and uniformly on $[-1,1]$.
\end{prop}
\begin{proof}
    Let $f(t) = \sum_{n=0}^\infty a_n P_n(t)$ be the series representation of $f$ w.r.t.\ the Legendre polynomials of dimension $d$. By assumption, the sequence $(a_n)$ is non-negative. Therefore, we may define
    \begin{align*}
        g(t) := \frac{1}{\sqrt{\vol{d-1}}}\sum_{n=0}^\infty \sqrt{a_n N(d,n)}  P_n(t).
    \end{align*}
    Note that it is possible that the series does not possess a pointwise limit. In any case, $g\in\ltw$, since we can use (\ref{eq_LegendreNorm}) to obtain that the square of its norm is given by $\vol{d-2}^{-1}\sum_{n=0}^\infty a_n < \infty$, using the summability assumption on the coefficients of $f$. 

    The desired representation can then be inferred by an application of Lemma~\ref{lem_productIntegration}, from which we obtain that
    \begin{align*}
        \int_{\sph_{d-1}} g(\abrac{\eta, \xi}) g(\abrac{\zeta,\xi}) \dd\sigma(\xi)
        = \sum_{n=0}^\infty a_n P_n(\abrac{\eta,\zeta})
        = f(\abrac{\eta,\zeta}).
    \end{align*}
    The fact that these equalities may be understood pointwise and w.r.t.\ uniform convergence is implied by Lemma \ref{lem_CoeffSummable} using the summability assumption.
\end{proof}

\subsection{Representation of the RKHS of SFBM}\label{sec_repOfRKHSofSFBM}

In this subsection, we want to apply the existence of the integral representation for positive semi-definite kernels to obtain a representation of the RKHS of SFBM. Let us first recall what a reproducing kernel Hilbert space is.

\begin{defi}
    Let $X$ be an arbitrary non-empty set and let $\mathcal{F}(X)$ be the set of real-valued functions on $X$. A Hilbert space $\hb \subseteq \mathcal{F}(X)$ is called a reproducing kernel Hilbert space (RKHS) with reproducing kernel $K:X\times X \to \R$ if $K$ satisfies the reproducing properties
    \begin{align*}
        K_x &:= K(x,.) \in \hb, &\text{for any $x\in X$}\\
        f(x) &= \abrac{f, K_x}_\hb, &\text{ for any $x\in X$ and any $f\in\hb$}.
    \end{align*}
\end{defi}

This definition is taken from \cite[Definition 1.1]{Saitoh2016}. Theorem 2.2 in the same book states that for any positive semi-definite kernel $K$ there exists a unique RKHS admitting the reproducing kernel $K$. This fact produces a natural one-to-one correspondence between RKHS and centred Gaussian processes (cf.\ Bochner's Theorem), since they are both uniquely defined by a positive semi-definite kernel.

Functions in the RKHS allow us to perform change-of-measure arguments (cf.\ Lemma~\ref{lemRKHSBound} below). To this end, we require a function in the RKHS with some specific properties. Therefore, we first describe the RKHS of SFBM in terms of an integral representation, which requires the next auxiliary result. The proof of that result can be inferred from \cite{Istas2006} directly or by an application of Proposition~\ref{prop_sqIntegralRep} using non-negativity and summability of the coefficients shown in the proof of \cite[Theorem~1 and Theorem~2]{Istas2006}. 
The bound on the coefficients given in \cite{Istas2006} is, however, not sufficient for our purposes later on and we require a deeper analysis.

\begin{lem}\label{lem_SFBMKernelLegendreDecomp}
    Let $0<H\leq1/2$ and $d\geq 2$ be fixed. There exist coefficients $a_n^{(H)}$ such that
    \begin{align}
        1-\left(\frac{\arccos(t)}{\pi}\right)^{2H} = \sum_{n=0}^\infty \frac{\anbr{H}^2}{N(d,n)} P_n(t),\qquad t\in[-1,1], \label{eq_LegendreSeriesArccos}
    \end{align}
    where $(P_n)$ are the Legendre polynomials of dimension $d$. The sum converges uniformly on $[-1,1]$ and in $\ltw$.
\end{lem}

The representation given in the next corollary could also be inferred from Istas' $L_2$ representation of SFBM (cf.\ \cite[Chapter 4.2.2]{Istas12Manifold} and \cite{Istas2006}). Since the proof is very short and because it demonstrates how we treat integrals also later on, we include it here.

\begin{cor}\label{cor_kernelIntegralRep}
    Let $0<H\leq 1/2$, let $\eta,\zeta,\xi\in\sph_{d-1}$ and define
    \begin{align*}
        g_{H}(\abrac{\eta,\zeta}) & := \frac{1}{\sqrt{\vol{d-1}}} \sum_{n=0}^\infty  \an{H} P_n(\abrac{\eta,\zeta}), \\
        m_\eta^{(H)}(\xi) & := \frac{\pi^{H}}{\sqrt{2}} \left[g_{H}(\abrac{\eta,\xi}) - g_{H}(\abrac{O,\xi})\right],
    \end{align*}
    where the coefficients $(a_n^{(H)})$ are given in Lemma~\ref{lem_SFBMKernelLegendreDecomp}. Then the covariance function $\expec{S_H(\eta) S_H(\zeta)}$ of SFBM possesses the integral representation
    \begin{align} \label{eqn:intrepr4}
        \expec{S_H(\eta) S_H(\zeta)} = \int_{\sph_{d-1}} m_\eta^{(H)}(\xi) \; m_\zeta^{(H)}(\xi) \dd\sigma(\xi).
    \end{align}
\end{cor}
\begin{proof}
    Recall that $d(\eta,\zeta) = \arccos(\abrac{\eta,\zeta})$. Lemma~\ref{lem_SFBMKernelLegendreDecomp} and Lemma~\ref{lem_productIntegration} let us infer that
    \begin{align*}
        \int_{\sph_{d-1}} \pi^{2H} g_{H}(\abrac{\zeta,\xi}) g_{H}(\abrac{\eta,\xi}) \dd \sigma(\xi)
        & = \pi^{2H} \sum_{n=0}^\infty \frac{\anbr{H}^2}{N(d,n)} P_n(\abrac{\eta,\zeta})\\
        & = \pi^{2H} \left( 1 - \left(\frac{\arccos(\abrac{\eta,\zeta})}{\pi}\right)^{2H} \right)\\
        & = \pi^{2H} - d(\eta,\zeta)^{2H}.
    \end{align*}
    This implies that
    \begin{align*}
        & \int_{\sph_{d-1}} m_\eta^{(H)}(\xi) \; m_\zeta^{(H)}(\xi) \dd\sigma(\xi) \\
        & = \frac{1}{2} \int_{\sph_{d-1}} \pi^{2H} \left[g_{H}(\abrac{\eta,\xi}) - g_{H}(\abrac{O,\xi})\right] \left[g_{H}(\abrac{\zeta,\xi}) - g_{H}(\abrac{O,\xi})\right] \dd \sigma(\xi)\\
        & = \frac{1}{2}\left( \left[\pi^{2H} - d(\eta,\zeta)^{2H}\right] - \left[\pi^{2H} - d(\eta,O)^{2H}\right] - \left[\pi^{2H} - d(\zeta,O)^{2H}\right] + \left[\pi^{2H} - d(O,O)^{2H}\right]  \right)\\
        & = \expec{S_H(\eta) S_H(\zeta)}.\qedhere
    \end{align*}
\end{proof}

Given the representation of the covariance of SFBM in (\ref{eqn:intrepr4}), we can now give a representation of the RKHS of SFBM.

\begin{lem}\label{lemRKHSofSFBM}
    Let $m_\eta^{(H)}$ be defined as in Corollary~\ref{cor_kernelIntegralRep}. The RKHS of SFBM is given by
    \begin{align*}
        \hb := \left\{ h:\sph_{d-1} \mapsto \R \;\middle| \; h(\eta) = \int_{\sph_{d-1}} \ell(\xi) m_\eta^{(H)}(\xi) \dd \sigma(\xi),\; \ell\in L_2(\sph_{d-1}) \right\}.
    \end{align*}
    The norm on $\hb$ is bounded by
    \begin{align}\label{eq_RKHSNormBoundEll}
        \norm{h}_{\hb}^2 \leq \int_{\sph_{d-1}} \abs{\ell(\xi)}^2 \dd \sigma(\xi),
    \end{align}
    if $h$ can be written as the integral over $\ell$ like in the definition of $\hb$.
\end{lem}

\begin{proof}
    The proof is a slight modification of \cite[Example 4.5]{lifshits12lectures}, which is an application of Theorem 4.1 in the same book. We only need a factorisation of our covariance function into an integral of type (\ref{eqn:intrepr4}), 
    which was shown in Corollary~\ref{cor_kernelIntegralRep}. We are then guaranteed that if $J$ is the operator defined by
    \begin{align*}
        (J \ell)(\eta) := \int_{\sph_{d-1}} \ell(\xi) m_\eta^{(H)}(\xi) \dd \sigma(\xi),
    \end{align*}
    which maps $L_2$-functions defined on the sphere to real functions on the sphere, then the RKHS is given by
    \begin{align*}
        J(L_2(\sph_{d-1})) = \left\{ h:\sph_{d-1} \mapsto \R \;\middle| \; h(\eta) = \int_{\sph_{d-1}} \ell(\xi) m_\eta^{(H)}(\xi) \dd \sigma(\xi),\; \ell\in L_2(\sph_{d-1}) \right\}.
    \end{align*}
    Furthermore, the norm (cf.\ \cite[Thm. 4.1]{lifshits12lectures}) is equal to
    \begin{align*}
        \norm{h}_{\hb}^2 = \inf_{\substack{J\ell = h \\ \ell \in L_2}} \; \int_{\sph_{d-1}} \abs{\ell(\xi)}^2 \dd\sigma(\xi). \qquad\qquad \qedhere
    \end{align*}
\end{proof}

\section{Polynomial approximations}\label{sec_polynomialApprox}
This entire section is self-contained. The goal is to derive precise first-order asymptotics for the Legendre coefficients of $\ltw$ functions that are smooth inside the domain and that possess mixed algebraic endpoint singularities. Our main theorem of this section, Theorem~\ref{thm_LegendreCoeffOrder}, partly generalises the results in \cite{Sidi09} to Legendre polynomials of dimension $d$. In Section~\ref{sec_existencePartFktRKHS} we apply this result to construct a specific function in the RKHS of SFBM.

\subsection{Tools from Approximation Theory}\label{sec_toolsFromApprox}

We first require a result about polynomial approximations, commonly referred to as one of Jackson's theorems (cf.\ \cite[Introduction]{Sidi09}). This particular theorem is a slight modification of \cite[Theorem 6 in Chapter 5, p.\ 75]{lorentz86}.

\begin{lem}\label{lem_hoelderContImpliesLegendreDecay}
    Let $f:[-1,1]\to\R$ be a continuous function with $k\geq 0$ continuous derivatives $(f^{(0)} := f), f^{(1)}, \ldots, f^{(k)}$ and where $f^{(k)}$ is $\alpha_f$-H\"older continuous wit $0<\alpha_f<1$, i.e.\ there is a constant $M$, such that
    \begin{align*}
        \abs{f^{(k)}(x) - f^{(k)}(y)} \leq M \abs{x-y}^{\alpha_f},\qquad x,y\in[-1,1].
    \end{align*}
    
    Then there is a sequence of polynomials $(Q_n)_{n\in\N}$ with $\deg Q_n \leq n$, for which
    \begin{align*}
        \sup_{x\in[-1,1]} \abs{f(x) - Q_n(x)} \leq c \; n^{-k-\alpha_f}
    \end{align*}
    for any $n=1,2, \ldots$ and some constant $c > 0$ which depends on $f$ and $k$.
\end{lem}

We need a result that holds for any set of orthogonal polynomials (cf.\ \cite[Proposition 3.3]{frye12}), which we merely cite.

\begin{prop}\label{prop_orthogonalPolywrtSmallDeg}
    Let $(Q_n)$ be a set of orthogonal polynomials w.r.t.\ some weight function $w$ and let $\deg(Q_n) = n$. Then, for any polynomial $R$ with $\deg(R) < n$,
    \begin{align*}
        \int Q_n(x) R(x) w(x) \dd x = 0.
    \end{align*}
\end{prop}

To investigate the Legendre series coefficients of a larger class of functions $f$, we first consider the simplest case $f^{\pm}(x):=(1\pm x)^{\gamma}$. The coefficients w.r.t.\ Gegenbauer polynomials for non-integer $\gamma$ were calculated in \cite[Lemma A.1]{wang16}. Our calculation is a slight extension, because we are able to consider all possible parameters (also cf.\ Remark \ref{rem_gradsteynProof}).
We state the result in terms of our notation and, most importantly, derive the asymptotic behaviour.

\begin{lem}\label{lem_LegendreCoeffFpm}
    Let $f^{\pm}:[-1,1]\to\R$ be the functions defined by $f^\pm(x) := (1 \pm x)^{\gamma}$ with some $\gamma > \frac{1-d}{2}$. Then we have the following properties:
    \begin{itemize}
        \item[(i)]
    For $\gamma \notin \N_0$ the Legendre series coefficients $a_{\gamma,n}^-$ of $f^-$ are given by
    \begin{align*}
        a_{\gamma,n}^- := N(d,n) ~ \frac{\Gamma(\frac{d}{2})}{\sqrt{\pi}} ~ \frac{2^{d-2+\gamma} ~ \Gamma(\gamma + \frac{d-1}{2})}{\Gamma(\gamma+d+n-1)} ~ \frac{\Gamma(n-\gamma)}{\Gamma(-\gamma)}.
    \end{align*}
    The Legendre series coefficients $a_{\gamma,n}^+$ for $f^+$ can be computed as $a_{\gamma,n}^+ := (-1)^n a_{\gamma,n}^-$. 
            \item[(ii)]
    The asymptotic behaviour is given by
    \begin{align*}
        a_{\gamma,n}^- 
        \begin{cases}
            \sim \frac{c}{\Gamma(-\gamma)} \; n^{-2\gamma-1}, & \text{ for $\gamma \notin \N_0$}, \\
            = 0, & \text{ for $\gamma \in \N_0$ and $n > \gamma$},
        \end{cases}
    \end{align*}
    for some constant $c>0$ depending on $\gamma$ and $d$.
\end{itemize}
\end{lem}
\begin{rem}\label{rem_gradsteynProof}
\begin{itemize}
\item[(1)] From the proof of Lemma~\ref{lem_LegendreCoeffFpm} one may obtain the values of $a_{\gamma,n}^-$ (and, by symmetry, $a_{\gamma,n}^+$) for $\gamma\in\N_0$ through substituting the quotient $\frac{\Gamma(n-\gamma)}{\Gamma(-\gamma)}$ by $0$ if $\gamma \leq n-1$ and by $(-1)^n \frac{\Gamma(\gamma+1)}{\Gamma(\gamma-n+1)}$ if $\gamma \geq n$.
\item[(2)] The assumption $\gamma > \frac{1-d}{2}$ is necessary. Otherwise, $f^\pm$ are not $L_2^w$-functions and a Legendre decomposition does not exist. 
\end{itemize}
\end{rem}
\begin{proof}[Proof of Lemma~\ref{lem_LegendreCoeffFpm}.]
    We start with the calculation of the Legendre coefficients of $f^+$. By (\ref{legendre_orthogonal_decomposition}), we need to calculate
    \begin{align*}
        a_{\gamma,n}^+ = \frac{N(d,n) \vol{d-2}}{\vol{d-1}} \int_{-1}^1 (1-x)^{\frac{d-3}{2}} (1+x)^{\gamma+\frac{d-3}{2}} P_n(x) \dd x.
    \end{align*}
    Using $|P_n(t)|\leq 1$ one can show that the above integral exists for $\gamma > \frac{1-d}{2}$. 
    Since the generalised Legendre polynomials are re-scaled Gegenbauer polynomials, we are going to apply \cite[Eqn.\ 7.311(3)]{GradsteynRyzhik07}, from which we can infer that
    \begin{align}\label{eq_gradsteynFormulaPlain}
        \int_{-1}^1 (1-x)^{\nu-\eh} (1+x)^\beta C^{(\nu)}_n(x) \dd x = \frac{2^{\beta+\nu+\eh} \Gamma(\beta+1) \Gamma(\nu+\eh) \Gamma(2\nu+n)\Gamma(\beta-\nu+\frac{3}{2})}{n! ~ \Gamma(2\nu)\Gamma(\beta-\nu-n+\frac{3}{2}) \Gamma(\beta+\nu+n+\frac{3}{2})}
    \end{align}
    holds for $0\neq \nu > -\eh$ and $\beta > -1$ with $\beta-\nu+\frac{3}{2}\notin\Z_{\leq n}$. We proceed in several steps. In Step 1, we re-scale (\ref{eq_gradsteynFormulaPlain}) to fit our setting and show that one may apply continuous extension to obtain all required values of the integral. In Step 2, we calculate the continuous extension. In Step 3, we evaluate the quotient $\vol{d-2}/\vol{d-1}$ and calculate $a_{\gamma,n}^-$ from $a_{\gamma,n}^+$. 

    The calculation of the asymptotic behaviour requires (\ref{lem_orderOfNdn}) and is then an exercise in the repeated application of (\ref{eq_gammaQuotient}). The term $\Gamma(-\gamma)$ is the only factor (except for $(-1)^n)$) that determines the sign of the coefficient for $n$ large enough.

    \textit{Step 1: Continuous extension.} From Rodrigues' formula (\ref{eq_RodriguesGegenbauerDef}) for the definition of Gegenbauer polynomials one can infer that
    \begin{align}\label{eq_gegenbauerre-scaledContInParam}
        \frac{\Gamma(2\nu)~n!}{\Gamma(n+2\nu)} C_n^{(\nu)}(x)
    \end{align}
    is continuous in $0 \neq \nu > -\eh$ for any $x\in [-1,1]$. The continuity in $\pm 1$ follows from the re-scaling, since the term in (\ref{eq_gegenbauerre-scaledContInParam}) is constant $1$ for $x=1$ and $(-1)^n$ for $x=-1$. Due to the continuity of Rodrigues' formula in its parameter, we can infer that the generalised Legendre polynomials $(P_n^{(d)})$, where we stress the dependence on the dimension $d\geq 2$, are precisely given by
    \begin{align*}
        P_n^{(d)}(x) = \lim_{\substack{\nu'\to \frac{d}{2}-1 \\ \nu'\neq 0}} \frac{\Gamma(2\nu')\Gamma(n+1)}{\Gamma(n+2\nu')} C_n^{(\nu')}(x)
    \end{align*}
    even for $d=2$. Since $|P_n^{(d)}(x)|\leq 1$, we can choose any $d \geq 2$ and any sequence $\nu_k \to \frac{d}{2}-1$ with $\nu_k\neq 0$ and we know that the supremum
    \begin{align*}
        \sup_{k\in\N} \sup_{x\in [-1,1]} \frac{\Gamma(2\nu_k) n!}{\Gamma(n+2\nu_k)} \abs{C_n^{(\nu_k)}(x)} < \infty
    \end{align*}
    exists. This allows us to calculate the integral
    \begin{align*}
        \int_{-1}^1 (1-x)^{\frac{d-3}{2}} (1+x)^\beta P_n^{(d)}(x) \dd x
        = \lim_{\nu'\to \frac{d}{2}-1} \int_{-1}^1 (1-x)^{\nu'-\eh} (1+x)^\beta \frac{\Gamma(2\nu') n!}{\Gamma(n+2\nu')} C^{(\nu')}_n(x) \dd x
    \end{align*}
    through continuous extension. Therefore, by also substituting $\beta := \gamma + \frac{d-3}{2}$ we obtain the formula
    \begin{align}\label{eq_gradsteynToPnApplied}
        \int_{-1}^1 (1-x^2)^{\frac{d-3}{2}} (1+x)^\gamma P_n(x) \dd x
        = \frac{2^{d-2+\gamma} \Gamma(\gamma + \frac{d-1}{2}) \Gamma(\frac{d-1}{2}) }{ \Gamma(\gamma+d+n-1)} \frac{\Gamma(\gamma+1)}{\Gamma(\gamma-n+1)}
    \end{align}
    valid for any $\gamma > \frac{d-1}{2}$ with $\gamma \notin \Z_{\leq n-1}$. The latter restriction can be dealt with by noticing that for any convergent sequence $\gamma_k'\to \gamma > \frac{1-d}{2}$, where $\gamma_k'\notin\Z_{\leq n-1}$, the integrand on the LHS of (\ref{eq_gradsteynToPnApplied}) can be majorised, as well. Therefore, we also may apply continuous extension on the RHS in $\gamma$. The only term that is not defined for limits $\gamma\in\Z_{\leq n-1}$ is the quotient $\frac{\Gamma(\gamma+1)}{\Gamma(\gamma-n+1)}$, which we take care of in the next step.

    \textit{Step 2: Calculating the extension.} We proceed to calculate the limit $\lim_{\gamma'\to\gamma} \frac{\Gamma(\gamma'+1)}{\Gamma(\gamma'-n+1)}$ that comes up when calculating the extension on the RHS of (\ref{eq_gradsteynToPnApplied}). For integer $\gamma\geq n$, the limit is defined as it is. The limit is easily seen to be zero for $\gamma \in \{0,\ldots n-1\}$. These facts are used to infer the statement in Remark \ref{rem_gradsteynProof} (1). We are now going to derive a formula that works for all $\gamma \in\R\setminus\N_0$ with $\gamma > \frac{1-d}{2}$. Euler's reflection formula (cf.\ \cite[\href{https://dlmf.nist.gov/5.5.E3}{Equation 5.5.3}]{DLMF}) is given by $\Gamma(z)\Gamma(1-z) = \frac{\pi}{\sin(\pi z)}$, for $z\notin\Z$. Applying it twice lets us infer that
    \begin{align*}
        \Gamma(\gamma+1-n)\Gamma(n-\gamma) = \frac{\pi}{\sin(\pi (\gamma+1-n))} = (-1)^n \frac{\pi}{\sin(\pi(\gamma+1))} = (-1)^n \Gamma(1+\gamma) \Gamma(-\gamma)
    \end{align*}
    for $\gamma\notin\Z$. A rearrangement yields that for any $\gamma \notin \N_0$,
     \begin{align}\label{eq_gammaQuotientCalc}
        \lim_{\gamma'\to\gamma} \frac{\Gamma(\gamma'+1)}{\Gamma(\gamma'-n+1)} = (-1)^n \frac{\Gamma(n-\gamma)}{\Gamma(-\gamma)}.
    \end{align}

    \textit{Step 3: Final calculations.} To obtain the Legendre coefficients of $f^+$, we have yet to scale (\ref{eq_gradsteynToPnApplied}) by the normalising factor $N(d,n)\vol{d-2}/\vol{d-1}$. The spherical surface area $\vol{d-1}$ is well known (cf.\ \cite[Eq. 4.633]{GradsteynRyzhik07}) to be equal to
        $\vol{d-1} =  2 \pi^{\frac{d}{2}}/\Gamma\left(\frac{d}{2}\right)$.
    We can therefore calculate the quotient
        $\vol{d-2}/\vol{d-1} =  \Gamma(\frac{d}{2})/(\sqrt{\pi}\Gamma(\frac{d-1}{2}))$.
    With the $\Gamma(\frac{d-1}{2})$ term in this quotient and in (\ref{eq_gradsteynToPnApplied}) cancelling out, we have thus calculated $a_{\gamma,n}^+$ to be equal to the value given in the lemma. To calculate $a_{\gamma,n}^-$ note that $P_n(-x) = (-1)^n P_n(x)$. The substitution $x\mapsto -x$ yields that
    \begin{align*}
        \int_{-1}^1 (1-x^2)^{\frac{d-3}{2}} (1+x)^\gamma P_n(x) \dd x
        = (-1)^n \int_{-1}^1 (1-x^2)^{\frac{d-3}{2}} (1-x)^\gamma P_n(x) \dd x,
    \end{align*}
    which explains the identity $a_{\gamma,n}^+ = (-1)^n a_{\gamma,n}^-$.
\end{proof}

\subsection{Main approximation result}\label{sec_mainApproxRes}

We would like to generalise our analysis of the Legendre coefficients to a larger class of functions. The assumptions below might seem very restrictive, but they are satisfied in many interesting cases (cf.\ \cite[Sec. 2]{Sidi09}). Briefly summarised, our main approximation theorem (Theorem~\ref{thm_LegendreCoeffOrder}) states that the asymptotic decay of the Legendre coefficients of a function that is smooth inside the interval can be determined by the behaviour around the endpoints. If $f$ is not smooth at one of the endpoints $\pm 1$ the term ``algebraic / logarithmic endpoint singularity'' is used in the literature for a classification of the problem (cf.\ \cite{Sidi09}). Due to our application, we only consider (mixed) algebraic endpoint singularities.

\begin{ass}\label{assumptionLegendreAsymp}
    Let $f:(-1,1) \to\R$ be a function, such that
    \begin{enumerate}
        \item \label{ass_cInf} $f$ is in $C^\infty(-1,1)$.
        \item \label{ass_ab} There exist real sequences $(A_i)$, $(B_i)$ and $(\alpha_i)$, $(\beta_i)$ that satisfy
        \begin{align*}
            \frac{1-d}{2} &< \alpha_0 < \alpha_1 < \alpha_2 < \ldots  & \text{ with } \lim_{i\to\infty} \alpha_i = \infty, \\
            \frac{1-d}{2} &< \beta_0 < \beta_1 < \beta_2 < \ldots & \text{ with } \lim_{i\to\infty} \beta_i = \infty,
        \end{align*}
     such that we have the asymptotic expansions
        \begin{align*}
            f(t) & = \sum_{i=0}^\infty A_i (1+t)^{\alpha_i} & \text{ for } t \searrow -1, \\
            f(t) & = \sum_{i=0}^\infty B_i (1-t)^{\beta_i} & \text{ for } t \nearrow +1,
        \end{align*}
        in the sense of Remark \ref{rem_asympExpDef}. Moreover, if $A_j = 0$ for some $j\geq 0$, then we require that $A_i=0$ for all $i\geq j$; and anlogously for the sequence $(B_i)$.
        \item \label{ass_term-wiseDiff} The $k$-th derivative of $f$ near the singularities can be approximated by the term-wise $k$-th derivative of the series expansions in the sense of Remark \ref{rem_asympExpDef}, i.e.
        \begin{align*}
            f^{(k)}(t) & = \sum_{i=0}^\infty A_i \; \frac{\partial^k}{\partial t^k} (1+t)^{\alpha_i} & \text{ for } t \searrow -1, \\
            f^{(k)}(t) & = \sum_{i=0}^\infty B_i \; \frac{\partial^k}{\partial t^k} (1-t)^{\beta_i} & \text{ for } t \nearrow +1.
        \end{align*}
    \end{enumerate}
\end{ass}

\begin{rem}\label{rem_asympExpDef}
    The limit behaviour in Assumption~\ref{assumptionLegendreAsymp} \ref{ass_ab} is to be understood in the sense that
    \begin{align*}
        f(t) &= \sum_{i=0}^m A_i (1+t)^{\alpha_i} + \OO\left((1+t)^{\alpha_{m+1}}\right),
    \end{align*}
    for any $m\geq 0$ and $t\searrow -1$; and analogously for $t\nearrow +1$.
\end{rem}

We are now ready to state the main theorem to obtain the asymptotic behaviour of Legendre coefficients of a function with mixed algebraic end point singularities. We generalise the statement of \cite[Theorem 2.2]{Sidi09}  to a broader class of polynomials, as \cite{Sidi09} only considers classical Legendre polynomials. We note however that our conclusion is weaker as we are only concerned with the first order terms. One could likely generalise the following statement to the analysis of Jacobi polynomial coefficients and obtain the full asymptotics analogously to \cite{Sidi09}.

\begin{thm}\label{thm_LegendreCoeffOrder}
    Let $f:(-1,1) \to\R$ be a function that satisfies Assumption~\ref{assumptionLegendreAsymp}. Define 
    \begin{align*}
        \gamma_\alpha &:= \min\{\alpha_i :~ i\geq 0,~ \alpha_i \notin \N_0,~ A_i \neq 0\},\\
        \gamma_\beta &:= \min\{\beta_i :~ i\geq 0,~ \beta_i \notin \N_0,~ B_i \neq 0\},\\
        \gamma &:= \min\{\gamma_\alpha, \gamma_\beta\},
    \end{align*}
    and choose the value $\infty$ whenever the minimum is taken over the empty set. Let $(b_n)$ be the Legendre coefficients of the function $f$.
    \begin{enumerate}
        \item[(i)] 
    If $\gamma = \infty$ then $(b_n)$ decays faster than any polynomial.
        \item[(ii)]
            If $\gamma_\alpha \neq \gamma_\beta$ then for some constant $c>0$ depending on $f$ and $d$,
    \begin{align*}
        \abs{b_n} \sim c \; n^{-2\gamma- 1}.
    \end{align*}
    \item[(iii)]
            If $\gamma_\alpha = \gamma_\beta < \infty$ then let $j_\alpha$ and $j_\beta$, respectively, be the indices at which $\alpha_{j_\alpha}=\gamma=\beta_{j_\beta}$. Then
    \begin{align*}
        \abs{b_{n_k}} &\in o\left(n_k^{-2\gamma - 1}\right) & \text{ as $k\to\infty$ for any subsequence $(n_k)$ with $(-1)^{n_k} A_{j_\alpha} + B_{j_\beta} = 0$},\\
        \abs{b_{n_k}} &\sim c\; n_k^{-2\gamma - 1} & \text{ as $k\to\infty$ for any subsequence $(n_k)$ with $(-1)^{n_k} A_{j_\alpha} + B_{j_\beta} \neq 0$.}
    \end{align*}
    The constant $c>0$ may depend on $f$ and $d$.
    \end{enumerate}
\end{thm}

\begin{proof}   
    To simplify notation, let us define that a function is in $C^{\gamma'}$ for $\gamma'>0$, if the function is in $C^{\lfloor \gamma' \rfloor}$, and if $\gamma'$ is not an integer, then its $\lfloor \gamma' \rfloor$-th derivative is $\gamma' - \lfloor \gamma' \rfloor$ H\"older continuous.
    
    Let us first deal with the case in which $\gamma < \infty$. Choose $\gamma' > 0$ such that $2\gamma + d/2 < \gamma'$.  Choose $m_{\gamma'}\in\N$ such that $\gamma' < \min\{\alpha_{m_{\gamma'}}, \beta_{m_{\gamma'}}\}$. Such an index exists, because the sequences $(\alpha_i)$ and $(\beta_i)$ converge to infinity due to Assumption~\ref{assumptionLegendreAsymp} \ref{ass_ab}. In this case, the function $g:[-1,1]\to\R$ defined by
    \begin{align*}
        g(t) := f(t) - \sum_{i=0}^{m_{\gamma'}} A_i (1+t)^{\alpha_i} - \sum_{i=0}^{m_{\gamma'}} B_i (1-t)^{\beta_i}
    \end{align*}
    is in $C^{\gamma'}[-1,1]$: Indeed, let $k:=\lfloor \gamma' \rfloor$. Then $g$ is $k$-times continuously differentiable, since clearly $g\in C^{\infty}(-1,1)$ and since
    \begin{align*}
        g^{(k)}(t) & = f^{(k)}(t) - \sum_{i=0}^{m_{\gamma'}} A_i \frac{\partial^k}{\partial t^k}(1+t)^{\alpha_i} - \sum_{i=0}^{m_{\gamma'}} B_i \frac{\partial^k}{\partial t^k}(1-t)^{\beta_i} & \\
        & \in o\left(\frac{\partial^k}{\partial t^k}(1+t)^{\alpha_{m_{\gamma'}}}\right)
         \subseteq o \left(\frac{\partial^k}{\partial t^k}(1+t)^{\gamma'}\right), & \text{ as $t\searrow -1$ } ,
    \end{align*}
    and analogously for $t\nearrow +1$. The required H\"older continuity of $g^{(k)}$ can be inferred from this calculation, as well. We can now split the Legendre coefficients $(b_n)$ of $f$ into known terms and a remainder term, i.e.
    \begin{align*}
        b_n = \sum_{i=0}^{m_{\gamma'}} A_i \; a_{\alpha_i,n}^+ + \sum_{i=0}^{m_{\gamma'}} B_i \; a_{\beta_i,n}^- + \frac{N(d,n)\vol{d-2}}{\vol{d-1}} \int_{-1}^1 g(t) P_n(t) w(t) \dd t,
    \end{align*}
    using notation from Lemma~\ref{lem_LegendreCoeffFpm}. To calculate the remainder, let $R_{n-1}$ be some polynomial of degree less than or equal to $n-1$. By Lemma~\ref{prop_orthogonalPolywrtSmallDeg},
    \begin{align*}
        \int_{-1}^1 g(t) P_n(t) w(t) \dd t
        & = \int_{-1}^1 \left[g(t)-R_{n-1}(t)\right] P_n(t) w(t) \dd t.
    \end{align*}
    An application of the Cauchy-Schwarz inequality in $L_2^w[-1,1]$ yields
    \begin{align}
        \left(\int_{-1}^1 g(t) P_n(t) w(t) \dd t\right)^2
        & = \left( \int_{-1}^1 \left[g(t)-R_{n-1}(t)\right] P_n(t) w(t) \dd t \right)^2 \notag \\
        & \leq \int_{-1}^1 \left[g(t)-R_{n-1}(t)\right]^2 w(t) \dd t \cdot \int_{-1}^1 P_n(t)^2 w(t) \dd t \notag \\
        & = \int_{-1}^1 \left[g(t)-R_{n-1}(t)\right]^2 w(t) \dd t \cdot \frac{\vol{d-2}}{N(d,n)\vol{d-1}}. \label{eqn:ubdelta}
    \end{align}
    Since $g\in C^{\gamma'}[-1,1]$, Lemma~\ref{lem_hoelderContImpliesLegendreDecay} implies that there is a sequence of polynomials $(R_n)_{n\in\N}$ with $\operatorname{deg}(R_n)\leq n$ such that
    \begin{align*}
        \sup_{t\in[-1,1]} \abs{g(t)-R_{n-1}(t)} \leq c\; n^{-\gamma'}
    \end{align*}
    for all $n\in\N$ and some constant $c>0$ that depends on $g$. This estimate allows us to infer that
    \begin{align*}
        \int_{-1}^1 \left[g(t)-R_{n-1}(t)\right]^2 w(t) \dd t
        & \leq \sup_{t\in[-1,1]} \abs{g(t)-R_{n-1}(t)}^2 \int_{-1}^1 w(t) \dd t \\
        & \leq c\; n^{-2\gamma'} \int_{-1}^1 P_0(t)^2 w(t) \dd t \\
        & = c\; n^{-2\gamma'} \frac{\vol{d-2}}{\vol{d-1}},
    \end{align*}
    where we used that $P_0 \equiv 1$. Combining this with (\ref{lem_orderOfNdn}) and (\ref{eqn:ubdelta}), we have
    \begin{align*}
        \abs{ \frac{N(d,n)\vol{d-2}}{\vol{d-1}} \int_{-1}^1 g(t) P_n(t) w(t) \dd t}
        & \leq c N(d,n)^{1/2} n^{-\gamma'}
        \leq c \; n^{-\gamma' +\frac{d-2}{2}},
    \end{align*}
    where $c>0$ is some constant depending on $d$ and $f$. Using Lemma~\ref{lem_LegendreCoeffFpm}, we know that the $a_{\alpha_i,n}^+$ which decays the slowest is the one with the smallest exponent $\alpha_i \notin\N_0$, i.e.\ $a_{\gamma_\alpha,n}^{+}$ decays the slowest. In the case $\gamma_\alpha = \infty$, ignore $a_{\gamma_\alpha,n}^{+}$ in the following calculation. Analogously, this holds for $a_{\beta_i,n}^-$ and $a_{\gamma_\beta,n}^{-}$, respectively. Since we assumed that $\gamma < \infty$, at least one of $\gamma_\alpha$ and $\gamma_\beta$ must be finite. Thus, we have
    \begin{align*}
        b_n & = \sum_{i=0}^{m_{\gamma'}} A_i \; a_{\alpha_i,n}^+ + \sum_{i=0}^{m_{\gamma'}} B_i \; a_{\beta_i,n}^- + N(d,n) \int_{-1}^1 g(t) P_n(t) w(t) \dd t \\
        & \sim c_1 (-1)^n \frac{A_{\gamma_\alpha}}{\Gamma(-\gamma_\alpha)} n^{-2\gamma_{\alpha}-1} + c_2 \frac{B_{\gamma_\beta}}{\Gamma(-\gamma_\beta)} n^{-2\gamma_{\beta}-1} + \OO\left( n^{-\gamma' +\frac{d-2}{2}} \right).
    \end{align*}
    Using the fact that $a_{\gamma,n}^+ := (-1)^n a_{\gamma,n}^-$, which is given by Lemma~\ref{lem_LegendreCoeffFpm}, we can infer that $c_1$ and $c_2$ are identical. Recall that we chose $\gamma'$ to satisfy $2\gamma + d/2 < \gamma'$, which implies that $-2\min\{\gamma_\alpha, \gamma_\beta\}-1 > -\gamma'+\frac{d-2}{2}$, such that the remainder vanishes faster than the other terms. From this observation, we obtain all statements about the case $\gamma < \infty$. In the case $\gamma = \infty$, we can simply choose $\gamma'$ as high as we like. The remainder term is the only one contributing to the asymptotic behaviour, because all non-negative integer contributions vanish for $n$ large enough. Therefore, we surpass any polynomial decay.
\end{proof}

We would like to state one general fact about power series that is going to make it less messy to check Assumption \ref{assumptionLegendreAsymp} in the subsequent section.

\begin{lem}\label{lem_powerSeries}
    Fix $\gamma > 0$ and $u\in\R$. Let
    \begin{align*}
        p(x) := \sum_{i=0}^\infty a_i x^i, \qquad
        q(x) := \sum_{j=0}^\infty b_j (x-u)^{\gamma + j}
    \end{align*}
    be power series, such that $q$ converges absolutely on some interval $[u,u+s)$ for some $s>0$. Furthermore, if $\widehat{q}(x) := \sum_{j=0}^\infty \abs{b_j}(x-u)^{\gamma+j}$ then let the range $\widehat{q}([u,u+s))$ be contained in some interval $(-r,r)$ on which $p$ converges absolutely for some $r>0$. Then $p\circ q$ can be written as an absolutely convergent series of (fractional) powers that is arbitrarily often term-wise differentiable on $(u,u+s)$. If we replace $q$ by $\widetilde{q}$ defined by
    \begin{align*}
        \widetilde{q}(x) := \sum_{j=0}^\infty b_j (u-x)^{\gamma + j}
    \end{align*}
    and $[u,u+s)$ by $(u-s,u]$ in the above we obtain the analogous result.
\end{lem}
The proof of this result is a standard calculus procedure and can be found in the appendix.

\paragraph*{Remarks on related works.} There are many works dedicated to the study of coefficients of polynomial approximations. The closest to our result being \cite[Theorem 2.2]{Sidi09}, in which the full asymptotic behaviour of the classical Legendre coefficients for analytic functions with mixed algebraic endpoint singularities is calculated. In \cite[Theorem 2.3]{Sidi09}, logarithmic endpoint singularities are also considered. To name a few more, in \cite[Theorem 5.10 \& Remark 5.12]{wang16} the Gegenbauer coefficients of the basic functions $(1\pm t)^\gamma h(t)$ with $h$ being analytic on $[-1,1]$ are calculated. In \cite[Theorem 2.3]{zhaBoyd2021}, they calculate the precise first-order asymptotics of Chebyshev coefficients of functions of type $(1-t^2)^\gamma [\log(1-t^2)]^\theta h(t)$ with $h$ as before. In \cite[Theorem 4.1]{liuEtAl2024} upper bounds for coefficients of Jacobi polynomial approximations of functions in certain fractional Sobolev spaces are given. The major problem we had to deal with, which made none of the above results applicable (except for Sidi's in the case $d=3$), is the expansion of $\arccos(t)^{2H}$ at $-1$, which is not of type $(1+t)^\gamma h(t)$ with $h$ being analytic on $[-1,1]$ (cf.\ the proof of our Lemma \ref{lem_arccosSeriesRep}).

\section{Existence of a particular function in the RKHS}\label{sec_existencePartFktRKHS}

This chapter deals with the construction of the function in the RKHS using the tools prepared in Section \ref{sec_polynomialApprox}.

\subsection{Legendre coefficients of Arccos}\label{sec_legendreCoeffOfArccos}

In the following, we require the notion of a generalised binomial coefficient defined by
\begin{align*}
    \binom{x}{0} := 1, \qquad \binom{x}{i} := \frac{x(x-1)\ldots (x-i+1)}{i!}, \quad i=1,2,\ldots , \text{ for any }x\in\R.
\end{align*}

\begin{lem}\label{lem_arccosRepAtSing}
    The Arccos function satisfies Assumption~\ref{assumptionLegendreAsymp} with the representation
    \begin{align*}
        \arccos(t) &=&& \sum_{i=0}^\infty (-1)^i \frac{\binom{-\eh}{i}}{\eh+i} \left(\frac{1-t}{2}\right)^{\eh+i}  && \text{ as $t\nearrow +1$;} \\
        \arccos(t) &=& \pi - &\sum_{i=0}^\infty (-1)^i \frac{\binom{-\eh}{i}}{\eh+i} \left(\frac{1+t}{2}\right)^{\eh+i}  && \text{ as $t\searrow -1$.}
    \end{align*}
\end{lem}

Even though the function Arccos is very classical, we could not locate a reference in the literature that shows all properties from Assumption~\ref{assumptionLegendreAsymp}. Therefore, we give the proof of the lemma in the appendix. 

The following result is a precise estimate of the asymptotic behaviour of the Legendre coefficients of the Arccos function raised to a fractional power. A rough upper bound was obtained during calculations in \cite{Istas2006}. However, we need more precise upper and lower bounds during our construction in Lemma~\ref{lemRKHS_SFBM_fdelta} below.

\begin{lem}\label{lem_arccosSeriesRep}
    The coefficients $\an{H}$ defined in Lemma~\ref{lem_SFBMKernelLegendreDecomp} through the series representation (\ref{eq_LegendreSeriesArccos}) behave asymptotically as
    \begin{align*}
        \an{H} & \sim c ~ n^{-H+\frac{d-3}{2}} & \text{ for } 0<H<1/2 \text{ and } n \to\infty
    \end{align*}
    for some constant $c$ that depends on $H$ and $d$. In the case $H=1/2$, we have
    \begin{align*}
        \an{1/2} & = 0 & \text{ for even $n\geq 2$,} \\
        \an{1/2} & \sim c ~ n^{-\eh+\frac{d-3}{2}} & \text{ for odd $n\to\infty$},
    \end{align*}
    where the constant depends on $d$.
\end{lem}
\begin{proof}
    We need the Arccos function raised to a fractional power to satisfy Assumption~\ref{assumptionLegendreAsymp}. In particular, we need to establish the existence of a suitable series representation at the points $\pm 1$.

    Let us start with the behaviour at $+1$. Using the series representation of Arccos from Lemma~\ref{lem_arccosRepAtSing}, we know that
    \begin{align*}
        \arccos(t)^{2H} &= \left( \sum_{i=0}^\infty (-1)^i \frac{\binom{-\eh}{i}}{\eh+i} \left(\frac{1-t}{2}\right)^{\eh+i} \right)^{2H}, & \text{ as $t\nearrow +1$.}
    \end{align*}
    We can divide this expression by $\left(\frac{1-t}{2}\right)^H$ to obtain
    \begin{align*}
        \frac{\arccos(t)^{2H}}{\left(\frac{1-t}{2}\right)^H} &= \left( \sum_{i=0}^\infty (-1)^i \frac{\binom{-\eh}{i}}{\eh+i} \left(\frac{1-t}{2}\right)^{i} \right)^{2H}, & \text{ as $t\nearrow +1$.}
    \end{align*}
    Now note that the series on the RHS is a holomorphic function defined on some domain around $+1$ taken to a fractional power. Since the series is equal to $2$ at $t=+1$ (due to the $i=0$ term), we can choose the natural branch of the complex logarithm to infer that there exists a domain, on which the series raised to the power $2H$ is holomorphic around $t=+1$. Therefore, a representation of the following type must exist:
    \begin{align*}
        \arccos(t)^{2H}  = \left(\frac{1-t}{2}\right)^H \sum_{i=0}^\infty \widetilde{B}_i (1-t)^i
         = \sum_{i=0}^\infty B_i (1-t)^{H+i},
    \end{align*}
    which is term-wise differentiable on some domain to the left of $+1$. Note that the non-integer $H$ is the smallest exponent.
    
    The representation at $-1$ works quite differently. We again start with the series representation of Arccos at $-1$ from Lemma~\ref{lem_arccosRepAtSing} and write
    \begin{align*}
        q(t):=\frac{\arccos(t)}{\pi} - 1 = -\frac{1}{\pi} \sum_{i=0}^\infty (-1)^i \frac{\binom{-\eh}{i}}{\eh+i}\left(\frac{1+t}{2}\right)^{\eh+i}.
    \end{align*}
    Additionally, we consider the series representation (cf.\ \cite[Eq. 1.110]{GradsteynRyzhik07})
    \begin{align*}
        p(z):=1-(1+z)^{2H} = - \sum_{j=1}^\infty \binom{2H}{j} ~ z^j
    \end{align*}
    at $z=0$. For generalised binomial coefficients, we know that the asymptotic behaviour is given by
    \begin{align*}
        \binom{x}{i} = \frac{\Gamma(x+1)}{\Gamma(i+1) \Gamma(x-i+1)} = (-1)^i \frac{\Gamma(i-x)}{\Gamma(i+1)\Gamma(-x)} \sim c ~ \frac{(-1)^i}{\Gamma(-x)} i^{-x-1}
    \end{align*}
    for some constant $c>0$ using (\ref{eq_gammaQuotientCalc}) for $x\notin\N_0$. This reveals that $q$ and $p$ are absolutely convergent -- the first one for $t\in[-1,1]$, the second one for $z\in[-1,1]$. We want to apply Lemma~\ref{lem_powerSeries}. Since all coefficients of $q$ are negative and the range of $q$ is $[-1,0)$ for $t\in[-1,1)$, the range of $\widehat{q}([-1,1))$ (cf.\ notation in Lemma~\ref{lem_powerSeries}) is contained in $(-1,1)$, where $p$ is absolutely convergent. Thus, we obtain that the composed power series $p(q(.))$ exists on $[-1,1)$ and is term-wise differentiable on the interval $(-1,1)$. Therefore, all parts of Assumption~\ref{assumptionLegendreAsymp} are satisfied. Note that the non-integer number $\eh$ is the smallest exponent in the series for $p(q(.))$. 

    In case $H<\eh$, we can thus apply Theorem~\ref{thm_LegendreCoeffOrder} with $\gamma_\beta = H < \eh$ and $\gamma_\alpha = \eh$ to obtain the asymptotics of the Legendre coefficients $(b_n)$, which by (\ref{eq_LegendreSeriesArccos}) are related to the coefficients $(a_n^{(H)})$ through $b_n=\anbr{H}^2/N(d,n)$. Using (\ref{lem_orderOfNdn}), we thus obtain that $\an{H} \sim c \, n^{-H+\frac{d-3}{2}}$, for any $H<1/2$.

    In case $H=\eh$, we may apply Theorem~\ref{thm_LegendreCoeffOrder} with $\gamma_\alpha = \gamma_\beta = \eh$ and $A_{j_\alpha} = - B_{j_\beta} = \sqrt{2}$. This lets us derive the asymptotic behaviour for odd $n$. For even $n \geq 2$, note that
    \begin{align*}
        \int_{-1}^1 \left(1-\frac{\arccos(t)}{\pi}\right) P_n(t) ~w(t)~\dd t = \int_{-1}^1 \left(\eh - \frac{\arccos(t)}{\pi}\right) P_n(t) ~w(t)~\dd t = 0,
    \end{align*}
    due to Proposition~\ref{prop_orthogonalPolywrtSmallDeg} combined with $n>0$, and because the integrand on the RHS is antisymmetric.
\end{proof}

Concerning the behaviour of the coefficients of the series representation of Arccos, we need to be even more specific and prove that none of the odd Legendre series coefficients is \textit{ever} zero. This is important, since we shall divide by these coefficients later on. The strict positivity is proven in Lemma~\ref{lem_arccosFracNonZeroCoeff}, for which we need the following auxiliary result.

\begin{lem}\label{lem_monotoneMonomials}
    Let $n,m\in \N_0$. Then
    \begin{align*}
        \int_{-1}^1 x^n (P_{m}(x) - P_{m+2}(x)) w(x) \dd x \geq 0.
    \end{align*}
\end{lem}
\begin{proof}
    Recall that the polynomials $P_m$ are even, whenever $m$ is even, and odd, whenever $m$ is odd. Thus, if $m$ is even and $n$ is odd or if $m$ is odd and $n$ is even, then the integral is zero and the statement is true. If $n < m$, the statement follows from Proposition \ref{prop_orthogonalPolywrtSmallDeg}. 
    
    Formula \cite[7.311(2)]{GradsteynRyzhik07} calculates the integral from $0$ to $1$ of products of Gegenbauer polynomials together with monomials w.r.t.\ the weight $w$. Using the same type of continuous extension argument as in the proof of Lemma~\ref{lem_LegendreCoeffFpm} and symmetry, which requires us to multiply by $2$ to obtain the integral from $-1$ to $1$, we can infer that
    \begin{align}\label{eq_gradsteynPolynomial}
        \int_{-1}^1 x^n P_{m}(x) w(x) \dd x
        & = \frac{\Gamma(n+1) \Gamma(\frac{d-1}{2}) \Gamma(\frac{n-m+1}{2}) }{2^m \Gamma(n-m+1) \Gamma(m+\frac{n-m+d}{2})}
    \end{align}
    holds for $n\geq m$ with $m,n$ both even or both odd. If $n=m$ then we immediately obtain that
    \begin{align*}
        \int_{-1}^1 x^m (P_{m}(x) - P_{m+2}(x)) w(x) \dd x = \int_{-1}^1 x^n P_{n}(x) w(x) \dd x > 0.
    \end{align*}
    Thus, let us assume that $n\geq m+2$. Using (\ref{eq_gradsteynPolynomial}), we can calculate that
    \begin{align*}
        \int_{-1}^1 & x^n (P_{m}(x) - P_{m+2}(x)) w(x) \dd x \\
        & = \frac{ \Gamma(n+1) \Gamma(\frac{d-1}{2}) \Gamma(\frac{n-m+1}{2} - 1) }{2^{m+2} \Gamma(n-m+1) \Gamma(m+\frac{n-m+d}{2}+1)} \\
        & \qquad \cdot \left[ 4 \left(m+\frac{n-m+d}{2} \right) \left(\frac{n-m+1}{2}-1\right) - (n-m+1-2)(n-m+1-1) \right] \\
        & = \frac{ \Gamma(n+1) \Gamma(\frac{d-1}{2}) \Gamma(\frac{n-m+1}{2} - 1) }{2^{m+2} \Gamma(n-m+1) \Gamma(m+\frac{n-m+d}{2}+1)} \left[ (n-m-1)(2m+d) \right] > 0,
    \end{align*}
    using the assumption that $n\geq m+2$.
\end{proof}

\begin{lem}\label{lem_arccosFracNonZeroCoeff}
    The coefficients $\an{H}$ defined in Lemma~\ref{lem_SFBMKernelLegendreDecomp} through the series representation in (\ref{eq_LegendreSeriesArccos}) are strictly positive for $n=0$ and any odd $n$.
\end{lem}
\begin{proof}
    The coefficient $a_0^{(H)}$ is non-zero, since $P_0 \equiv 1$ for any $d$ and $1-\left(\frac{\arccos(x)}{\pi}\right)^{2H}$ is non-negative, continuous and non identical to the zero function. To obtain the other coefficients we need more background.

    We already know from Lemma~\ref{lem_arccosSeriesRep} that the odd coefficients are non-negative and converge to zero with some polynomial rate. We only need to show that the odd coefficients are monotone decreasing to imply the statement.

    The series representations
    \begin{align*}
        q(x) := \eh - \frac{\arccos(x)}{\pi} & = \frac{1}{\pi} \sum_{j=0}^\infty \frac{(2j)!}{(2^j j!)^2} \frac{x^{2j+1}}{2j+1}, & \text{ for $x\in [-1,1]$}, \\
        p(x) := 1-\left(\eh - z\right)^{2H} & = \left(1-2^{-2H}\right) + 2^{-2H} \sum_{i=1}^\infty (-1)^{i+1} \binom{2H}{i} (2 z)^i, & \text{ for $z\in \left[-\eh,\eh\right]$},
    \end{align*}
    are well-known: The expansion for Arccos and its radius of convergence can be inferred from (cf.\ \cite[\href{https://dlmf.nist.gov/4.24.E1}{Equation 4.21.1}]{DLMF}) together with its elementary relation to Arcsin. The representation of fractional powers and their series' absolute convergence on the given domain can be inferred from \cite[Eq. 1.110]{GradsteynRyzhik07}. 
    Therefore, both series define a power series. Additionally, since the range of $q$ is $[-1/2, 1/2]$ and thus compatible with $p$, we may compose them into the power series $p\circ q$ \begin{align*}
        1-\left(\frac{\arccos(x)}{\pi}\right)^{2H} & = \sum_{k=0}^\infty r_k x^k,
    \end{align*}
    where we note that $r_k \geq 0$. Indeed, the positivity of these coefficients follows by inserting both series into one another: The coefficients of $q$ are all positive and $\sign\left(\binom{2H}{i}\right) = (-1)^{i+1}$ for $i\geq 1$. Since additionally $1>2^{-2H}$, all coefficients of the composed power series are positive.
    
    Using the inequality given in Lemma~\ref{lem_monotoneMonomials}, Equation (\ref{eq_LegendreNorm}), the absolute convergence of the composed series, and the non-negativity of the coefficients we calculate that
    \begin{align*}
        \frac{\vol{d-1}}{\vol{d-2}} \frac{\anbr{H}^2}{N(d,n)^2} & = \int_{-1}^1 \left(1-\left(\frac{\arccos(x)}{\pi}\right)^{2H}\right) P_n(x) w(x) \dd x \\
        & = \int_{-1}^1 \sum_{k=0}^\infty r_k x^k P_n(x) w(x) \dd x \\
        & = \sum_{k=0}^\infty r_k \int_{-1}^1 x^k P_n(x) w(x) \dd x \\
        & \geq \sum_{k=0}^\infty r_k \int_{-1}^1 x^k P_{n+2}(x) w(x) \dd x \\
        & = \int_{-1}^1 \sum_{k=0}^\infty r_k x^k P_{n+2}(x) w(x) \dd x \\
        & = \frac{\vol{d-1}}{\vol{d-2}} \frac{\left[a_{n+2}^{(H)} \right]^2}{N(d,n+2)^2}.
    \end{align*}
    This implies that the odd coefficients of the Legendre series representation are non-increasing. Since we know that asymptotically $\frac{\anbr{H}^2}{N(d,n)}$ behaves like $c~n^{-2H-1}$ for odd $n$ due to Lemma~\ref{lem_arccosSeriesRep}, this term can never be zero. And since $N(d,n)$ is non-zero, neither can any of the odd $\an{H}$ be zero.
\end{proof}

\subsection{Construction of the RKHS function}\label{sec_constrOfRKHSFct}

In this subsection we are going to apply all the preparatory work to define a function in the RKHS with all of the properties that are necessary for our proof of the main theorem. Let us first define an auxiliary function.

\begin{lem}\label{lem_RKHSConstruction_auxFunction}
    Let $H>0$ and define $G^{(H)}:[-1,1]\to \R$ by
    \begin{align} \label{eqn:defnoffunctionfH}
        G^{(H)}(t) := 2^H - (1-t)^H + (1+t)^H
    \end{align}
    The function $\gh$ can be expanded in terms of Legendre polynomials of dimension $d$ as
        \begin{align}
            \gh(t) = \sum_{n=0}^\infty \frac{\bnbr{H}^2}{N(d,n)} P_n(t) \label{eq_LegendreSeriesFH}
        \end{align}
        with
        \begin{align*}
            \bn{H} &= 0, &\text{ for every even } n\geq 2; \\
            \bn{H} &\sim c \; n^{-H+\frac{d-3}{2}}, & \text{ for odd $n\to\infty$},
        \end{align*}
        and a constant $c$ which depends on $H$ and $d$. The representation (\ref{eq_LegendreSeriesFH}) converges uniformly on $[-1,1]$ and in $\ltw$.
\end{lem}
\begin{proof}
    Most of the work has already been done in Lemma~\ref{lem_LegendreCoeffFpm} as it allows us to infer that the $n$-th Legendre coefficient $b_n$ is given by
    \begin{align}\label{eq_coeffGH}
        b_0 &= 2^H a_{0,0}^{\pm}, & \notag\\
        b_n &= - a_{H,n}^- + a_{H,n}^+ = - a_{H,n}^- (1-(-1)^n), &\text{ for $n\geq 1$},
    \end{align}
    From this, we immediately obtain that $b_0=a_{0,n}^{\pm} > 0$, and that $b_n = 0$ for any even $n\in\N$. Furthermore, all factors in the definition of $a_{H,n}^-$ are positive, except for $\frac{1}{\Gamma(-H)}$, which is negative and which implies that, $b_n > 0$ for any odd $n$.

    The asymptotic behaviour and thus the absolute summability of the coefficients is implied by (\ref{eq_coeffGH}) together with the asymptotics given in Lemma~\ref{lem_LegendreCoeffFpm}.

    Thus, the representation (\ref{eq_LegendreSeriesFH}) exists by Proposition~\ref{prop_sqIntegralRep}. By the definition of  the sequence $(b_n^{(H)})$ in (\ref{eq_LegendreSeriesFH}), we have $b_n=\bnbr{H}^2/N(d,n)$ so that this together with (\ref{lem_orderOfNdn}) implies the properties of $(b_n^{(H)})$ as claimed.
\end{proof}

We can use the auxiliary function $G^{(H)}$ for the main construction in the next lemma. In the following, we denote by $\mathbb{H}$ the RHKS of SFBM $(S_H(\eta))$. Further, we use the following notation to refer to spherical-distance balls
\begin{align}\label{def_sphDistBall}
    \bb_\delta(\eta) := \left\{ \zeta\in\sph_{d-1} : d(\eta, \zeta) \leq \delta \right\}.
\end{align}
Later on, half-spheres are of particular interest. Therefore, we define
\begin{align}\label{def_halfSphere}
    \hh(\eta) := \{ \zeta\in\sph_{d-1} : d(\eta,\zeta) \leq \pi/2 \} = \bb_\ph(\eta)
\end{align}
as the half sphere centred at $\eta \in \sph_{d-1}$. 

\begin{lem}\label{lemRKHS_SFBM_fdelta} Fix $0<H\leq 1/2$ and $0<\alpha<1/2$. Then, for any $\delta>0$, the following function $\fda$ is in the RKHS $\mathbb{H}$ of SFBM $(S_H(\eta))$: 
    \begin{align*}
        f_{\delta,\alpha}(\eta) := 16^{(H+\alpha)} \left[ G^{((H+\alpha)/2)}(1) - G^{((H+\alpha)/2)}(\abrac{\eta,O}) \right] \delta^{-(H+\alpha)},
    \end{align*}
    where $G^{(H)}$ is the function defined in (\ref{eqn:defnoffunctionfH}).
    The function $\fda$ has the following properties:
    \begin{enumerate}
        \item $\delta^{2(H+\alpha)} \norm{\fda}_\mathbb{H}^2 \leq c$, i.e.\ the RKHS norm of $\delta^{H+\alpha} \fda$ is uniformly bounded in $\delta$ by some constant depending on $\alpha$, $H$, and $d$; \label{RKHS_SFBM_fda_d3_prop_normbound}
        \item $f_{\delta,\alpha}(\eta) \geq 1$ for all $\eta \in\sph_{d-1}\setminus \bb_\delta(O)$, where $\bb_\delta$ is as defined in (\ref{def_sphDistBall}). \label{RKHS_SFBM_fda_d3_prop_incr}
    \end{enumerate}
\end{lem}
\begin{proof}
    The proof is divided into several steps. The goal is to construct a function $\ell_{\delta,\alpha}^{(H)} \in L_2(\sph_{d-1})$, such that $\ell_{\delta,\alpha}^{(H)}$ integrated w.r.t.\ $m_\eta^{(H)}$ is equal to $\fda$, from which we can infer using Lemma~\ref{lemRKHSofSFBM} that $\fda$ is in the RKHS. We first write down the function $\ell_{\delta,\alpha}^{(H)}$, then we  show that it is in $L_2(\sph_{d-1})$. After this, we show that it yields the correct function when integrated w.r.t.\ the RKHS kernel. Lastly, we need to check the properties (a) and (b).

    \textit{Step 1: Preparation:} By Lemma~\ref{lem_RKHSConstruction_auxFunction}, we obtain a decomposition of $G^{((H+\alpha)/2)}$ into a basis of Legendre polynomials given in (\ref{eq_LegendreSeriesFH}) with coefficients $(b_n^{((H+\alpha)/2)})$. We also use the coefficients $(a_n^{(H)})$ from (\ref{eq_LegendreSeriesArccos}). Further, recall the functions
    \begin{align*}
        g_{H}(\abrac{\eta,\zeta}) &= \frac{1}{\sqrt{\vol{d-1}}} \sum_{n=0}^\infty  \an{H} P_n(\abrac{\eta,\zeta})\qquad \text{ and } \\
        m_\eta^{(H)}(\xi) &= \frac{\pi^{H}}{\sqrt{2}} \left[g_{H}(\abrac{\eta,\xi}) - g_{H}(\abrac{O,\xi})\right]
    \end{align*}
    from the representation of the RKHS in Corollary~\ref{cor_kernelIntegralRep}. Let $\N_0^\# := \{0\} \cup \{2k-1 : k\in\N \}$ denote the positive odd integers joined with $0$ and set
    \begin{align*}
        \ell_{\delta,\alpha}^{(H)}(\xi) := - \delta^{-(H+\alpha)} \frac{\sqrt{2} \pi^{-H} 16^{(H+\alpha)}}{\sqrt{\vol{d-1}}} \sum_{n\in\N_0^\#}  \frac{\bnbr{\hah}^2}{\an{H}} P_n(\abrac{O,\xi}).
    \end{align*}
    The unusual index set is used to avoid dividing by zero in the Brownian case $H=1/2$, since $\an{1/2} = 0$ for any $n\in \N \setminus\N_0^\#$. Note that $\an{H} \neq 0$ for any $n\in\N_0^\#$ and any $H\leq \eh$ by Lemma~\ref{lem_arccosFracNonZeroCoeff}. This index set is not a restriction, since $\bn{\hah} = 0$ for any $n\in \N \setminus\N_0^\#$, by Lemma~\ref{lem_RKHSConstruction_auxFunction}.
    
    \textit{Step 2: $L_2(\sph_{d-1})$ norm of $\ell_{\delta,\alpha}^{(H)}$:} To see that $\ell_{\delta,\alpha}^{(H)}$ is in $L_2(\sph_{2})$, we use the asymptotics of the coefficients of the Legendre series representation of the Arccos function raised to a fractional power obtained from Lemma~\ref{lem_arccosSeriesRep},  the asymptotics of $N(d,n)$ given in (\ref{lem_orderOfNdn}), and the asymptotics of the coefficients $b_n^{(H)}$ from Lemma~\ref{lem_RKHSConstruction_auxFunction}. Using Lemma~\ref{lem_productIntegration} in the first step, we obtain
    \begin{align}
        \int_{\sph_{d-1}} \ell_{\delta,\alpha}^{(H)}(\xi)^2 \dd\sigma(\xi)
        & = \delta^{-2(H+\alpha)} 2\pi^{-2H}  16^{2(H+\alpha)} \sum_{n\in\N_0^\#} \frac{\left[b_n^{((H+\alpha)/2)}\right]^4}{\anbr{H}^2 N(d,n)} \notag \\
        & \leq \delta^{-2(H+\alpha)} 2\pi^{-2H} 16^{2(H+\alpha)} \sum_{n\in\N^\#} c \; \; \frac{n^{-2(H+\alpha) + 2(d-3)}}{n^{-2H+d-3} n^{d-2}} \notag \\
        & = \delta^{-2(H+\alpha)} 2\pi^{-2H} 16^{2(H+\alpha)} \sum_{n\in\N^\#} c \; n^{-1-2\alpha} < \infty. \label{eqn:estimater34}
    \end{align}
    Note that we switched to $\N^\# := \N_0^\# \setminus\{0\}$ in the summation index to avoid dividing by zero. The constant may be increased to compensate for this. Furthermore, note that the constant may depend on $\alpha$, $H$ and $d$, but not on $\delta$. 
    
    \textit{Step 3: Integral of $\ell_{\delta,\alpha}^{(H)}$ w.r.t.\ the RKHS kernel:} Using the definition of $\ell_{\delta,\alpha}^{(H)}$ in the first step, the definition of $g_{H}$ in the second, Lemma~\ref{lem_productIntegration} in the third,
    and (\ref{eq_LegendreSeriesFH}) in the fourth, we obtain
    \begin{align*}
        \int_{\sph_{d-1}} & m_\eta^{(H)}(\xi) \ell_{\delta,\alpha}^{(H)}(\xi) \dd\sigma(\xi) \\
        & = - \delta^{-(H+\alpha)} \frac{\sqrt{2}\pi^{-H} 16^{(H+\alpha)}}{\sqrt{\vol{d-1}}} \frac{\pi^H}{\sqrt{2}}  \int_{\sph_{d-1}} \left[ g_{H}(\abrac{\eta,\xi} - g_{H}(\abrac{O,\xi}) \right]  \sum_{n\in\N_0^\#}  \frac{ \bnbr{\hah}^2 }{\an{H}} P_n(\abrac{O,\xi}) \dd\sigma(\xi) \\
        & = - \delta^{-(H+\alpha)} \frac{16^{(H+\alpha)}}{\vol{d-1}} \\
        & \qquad \cdot \int_{\sph_{d-1}} \left[ \sum_{n=0}^\infty  \an{H} P_n(\abrac{\eta,\xi}) - \sum_{n=0}^\infty  \an{H} P_n(\abrac{O,\xi}) \right]  \sum_{n\in\N_0^\#}  \frac{ \bnbr{\hah}^2 }{\an{H}} P_n(\abrac{O,\xi}) \dd\sigma(\xi) \\
        & =
        - \delta^{-(H+\alpha)} 16^{(H+\alpha)} \sum_{n\in\N_0^\#}  \frac{\bnbr{\hah}^2 }{N(d,n)} \left[ P_n(\abrac{O,\eta}) - P_n(\abrac{O,O})\right]\\
        & = - \delta^{-(H+\alpha)} 16^{(H+\alpha)} \left[ G^{(\hah)}(\abrac{O,\eta}) - G^{(\hah)}(1) \right] \\
        & = \delta^{-(H+\alpha)} 16^{(H+\alpha)} \left[ G^{(\hah)}(1) - G^{(\hah)}(\abrac{O,\eta}) \right] \\
        & = \fda(\eta) .
    \end{align*}
    Thus, $\fda$ is in the RKHS of SFBM. 

    \textit{Step 4: Verification of the properties.} 
    \ref{RKHS_SFBM_fda_d3_prop_normbound} We want to show that $\delta^{2(H+\alpha)} \norm{\fda}_\mathbb{H}^2$ is uniformly bounded in $\delta$. In (\ref{eqn:estimater34}), we calculated that (using also (\ref{eq_RKHSNormBoundEll}))
    \begin{align*}
        \norm{\fda}_\mathbb{H}^2 \leq \int_{\sph_{d-1}} \ell_{\delta,\alpha}^{(H)}(\xi)^2 \dd\sigma(\xi)
        & \leq \delta^{-2(H+\alpha)} 2\pi^{-2H} 16^{2(H+\alpha)} \sum_{n\in\N^\#} c \; n^{-1-2\alpha}.
    \end{align*}
    Multiplication by $\delta^{2(H+\alpha)}$ removes all dependence on $\delta$ and shows the statement.
    
    \ref{RKHS_SFBM_fda_d3_prop_incr} We want to show that $f_{\delta,\alpha}(\eta) \geq 1$ for all $\eta \in\sph_{d-1}\setminus \bb_\delta(O)$. Using elementary methods, one can verify that $1-\cos(x)\geq \frac{x^2}{16}$ for any $x\in[0,\pi]$. Thus,
    \begin{align}\label{eq_GHRoughEstimate}
        \delta^{-2H} \left( 2^H + (1-\cos(\delta))^H - (1+\cos(\delta))^H \right)
        \geq \delta^{-2H} \left( 2^H + \frac{\delta^{2H}}{16^{2H}} - 2^H \right) = 16^{-2H}.
    \end{align}
    For $\eta\in \sph_{d-1}\setminus \bb_\delta(O)$ we have $\abrac{O,\eta} \leq \cos{\delta}$ and therefore
    \begin{align*}
        \fda(\eta) & = 16^{(H+\alpha)} \left[ G^{((H+\alpha)/2)}(1) - G^{((H+\alpha)/2)}(\abrac{\eta,O}) \right] \delta^{-(H+\alpha)} \\
        & = 16^{(H+\alpha)} \left( 2^{(H+\alpha)/2} + (1-\abrac{\eta,O})^{(H+\alpha)/2} - (1+\abrac{\eta,O})^{(H+\alpha)/2} \right) \delta^{-(H+\alpha)} \\
        & \geq 16^{(H+\alpha)} \left( 2^{(H+\alpha)/2} + (1-\cos(\delta))^{(H+\alpha)/2} - (1+\cos(\delta))^{(H+\alpha)/2} \right) \delta^{-(H+\alpha)} \geq 1,
    \end{align*}
    using (\ref{eq_GHRoughEstimate}).
\end{proof}

\section{Proof of the main result}\label{sec_proofMain}

\subsection{Preliminaries}\label{sec_proofPrelim}

In this subsection, we introduce a fundamental result from Riemannian geometry. It is a geometric comparison inequality that we combine with a comparison inequality for stochastic processes to obtain lower bounds for the persistence probability.

The geometric comparison inequality in its full generality was published in \cite{toponogov59} by Toponogov in 1959. It is more accessibly stated as Theorem~12.2.2 in \cite{petersen2016} or as Theorem~IX.5.1 in \cite{chavel2006}. The full statement is phrased in terms of Riemann geometry and variable curvature, but since we are just dealing with the sphere, which has constant curvature, we only present a simplified version.

To state Toponogov's Theorem, we need to define comparison triangles. Let us first give a heuristic description.

\begin{figure}[!ht]
    \centering
    \includegraphics[scale=0.7]{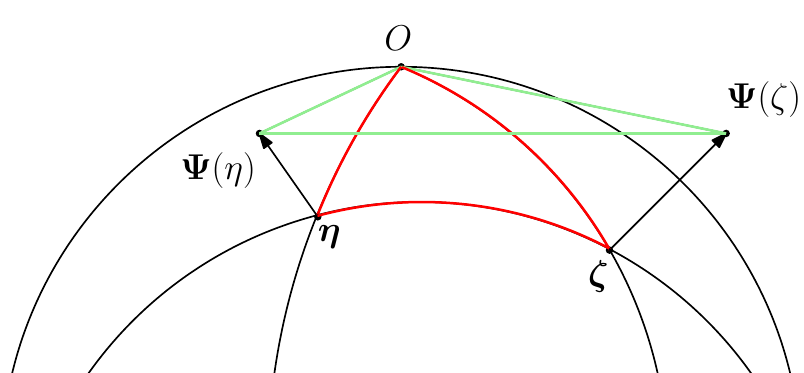}
    \caption{Comparison Triangles}
    \label{img_toponogov_lineGraphic}
\end{figure}

Imagine a Euclidean $(d-1)$-dimensional space tangential at the point $O$. In Figure \ref{img_toponogov_lineGraphic}, one can see a red spherical triangle, which consists of points $(\eta,O,\zeta)$. Using a function $\Psi$, we want to map it to the green triangle, which lies in the Euclidean space tangential to the sphere at $O$, and which consists of the points $(\Psi(\eta), 0, \Psi(\zeta))$. Note that we map $O$ to the origin $0\in\R^{d-1}$. We require the green triangle to preserve certain aspects of the red triangle: We want to preserve the lengths $\norm{\Psi(\eta)}_2 = d(\eta,O)$ and $\norm{\Psi(\zeta)}_2 = d(\zeta,O)$, and we want the angle of both triangles at $O$ and $0$, respectively, to be equal. The important question is: What happens to the distance $d(\eta,\zeta)$ compared to the distance $\norm{\Psi(\eta)-\Psi(\zeta)}_2$? This is answered by Toponogov's Theorem.

For more details on comparison triangles (also referred to as comparison hinges), see Chapter~11.2 in \cite{petersen2016}. The projection that we have just described heuristically is rigorously defined in Riemannian geometry as the inverse of the so-called exponential map $\exp_O: \R^{d-1} \to \sph_{d-1}$. For more general information on exponential maps, see \cite[Chapter~5.5.1]{petersen2016}.

We resort to an easier description of the projection in terms of linear algebra. Recall that the half-sphere centred at a point $\eta\in\sph_{d-1}$ is defined by $\hh(\eta)$, cf.\ (\ref{def_halfSphere}).

\begin{defi}\label{def_aeprojection} 
    The function $\Psi: \hh(O) \to \{x\in\R^{d-1} : \norm{x}_2\leq \pi/2\}$ is defined as follows:
    \begin{itemize}
        \item $\Psi(O) = 0 \in \R^{d-1}$
        \item Let $P:\operatorname{span}(O) \oplus \operatorname{span}(O)^\bot \to \R^{d-1}$ be the natural projection onto the orthogonal complement of the space spanned by $O$ denoted by $\operatorname{span}(O)^\bot \cong \R^{d-1}$. 
        
        Let $\eta\in\hh(O)\setminus O$. Then the space spanned by $O$ and $\eta$ in $\R^d$ is two-dimensional. The one-dimensional space orthogonal to $O$ in $\spann(\{O,\eta\})$ contains exactly two normalised vectors. Let $v_\eta'$ be the one with $\abrac{\eta,v_\eta'} > 0$ and define 
        \begin{align*}
            \Psi(\eta) := d(\eta,O)\; v_\eta 
        \end{align*}
        with $v_\eta := P v_\eta'$. Note that $\norm{v_\eta}_2 = \norm{v_\eta'}_2 = 1$, since $v_\eta'$ is orthogonal to $O$.
    \end{itemize}
    We call $\Psi$ the azimuthal equidistant projection.
\end{defi}

With the projection at hand, we can define triangles and comparison triangles.

\begin{defi}
    Let $(\eta,O,\zeta)$ be pairwise different with $\eta,\zeta\in\hh(O)$. We call $(\eta,O,\zeta)$ a spherical triangle. The triangle defined by $(\Psi(\eta), 0, \Psi(\zeta))$ with $\Psi$ as in Definition~\ref{def_aeprojection} in the Euclidean space $\R^{d-1}$ is called the comparison triangle w.r.t.\ $(\eta,O,\zeta)$.
\end{defi}

We are now ready to state our simplified version of Toponogov's Theorem.

\begin{thm}\label{thmToponogov}
    Given any triangle $(\eta, O, \zeta)$ on $\hh(O)$, the distance between the outlying points $d(\eta,\zeta)$ is smaller than that of the respective points in the comparison triangle $(\Psi(\eta),0, \Psi(\zeta))$, i.e.
    \begin{align*}
        d(\eta,\zeta) \leq \norm{\Psi(\eta) - \Psi(\zeta)}_2.
    \end{align*}
\end{thm}

In next lemma summarizes the properties of the projection $\Psi$ that we will need in the sequel.

\begin{lem}\label{lem_aepProperties}
    The azimuthal equidistant projection $\Psi$ given by Definition~\ref{def_aeprojection} has the following properties:
    \begin{enumerate}
        \item $\Psi$ is a bijection from $\hh(O)$ onto $\{x\in\R^{d-1} : \norm{x}_2 \leq \pi/2\}$, \label{lem_aepProperties_bij}
        \item $d(O,\eta) = \norm{\Psi(\eta)}_2$ for any $\eta\in\hh(O)$, \label{lem_aepProperties_isom}
        \item $d(\eta,\zeta) \leq \norm{\Psi(\eta)-\Psi(\zeta)}_2 $ for any $\eta,\zeta \in\hh(O)$. \label{lem_aepProperties_top}
    \end{enumerate}
\end{lem}
\begin{proof}
    One may verify the bijectivity using elementary methods, Property \ref{lem_aepProperties_isom} is evident from the construction, and Property \ref{lem_aepProperties_top} is a direct application of Theorem~\ref{thmToponogov}.
\end{proof}

The reason we need this geometric comparison theorem is to use it in combination with the following well-known stochastic comparison inequality. It is a statement about Gaussian processes and its discrete version is originally due to Slepian (cf.\ \cite{Slepian62}). The continuous-time analogue we are applying can be found in \cite[Lemma 1.2.5]{BaumgartenThesis2013}.
\begin{lem}[Slepian's lemma]\label{lem_Slepian}
    Let $(E,d)$ be a metric space and let $(X_t)_{t\in E}$ and $(Y_t)_{t\in E}$ denote two real-valued, separable, centred Gaussian processes, such that
    \begin{align*}
        \expec{X_t^2} & = \expec{Y_t^2}, \quad \forall t\in E,\\
        \expec{X_s X_t} & \leq \expec{Y_s Y_t}, \quad \forall s,t\in E.
    \end{align*}
    Let $f: E\to \R$ be a measurable function that is continuous except for countably many points. Then
    \begin{align*}
        \prob{X_t \leq f(t) \quad \forall t\in E} \quad \leq \quad \prob{Y_t \leq f(t) \quad \forall t\in E}.
    \end{align*}
\end{lem}

This yields an immediate corollary for positively correlated processes, which we frequently apply. The proof is standard; and we omit it.

\begin{cor}\label{cor_SlepianPosCorr}
    Let $(E,d)$ be a metric space and let $(Y_t)_{t\in E}$ denote a real-valued, separable, centred Gaussian process such that $\expec{Y_s Y_t} \geq 0$, $\forall s,t\in E$.
    Let $f: E\to \R$ be a measurable function that is continuous except for countably many points. Let $E = A \cup B$. Then
    \begin{align*}
        \prob{Y_t \leq f(t) ~ \forall t\in E} \quad \geq \quad \prob{Y_t \leq f(t) ~ \forall t\in A} \cdot \prob{Y_t \leq f(t) ~ \forall t\in B}.
    \end{align*}
\end{cor}

The previous corollary is particularly interesting to us, since SFBM is non-negatively correlated, as the next lemma shows. The proof of the lemma is a straightforward application of the triangle inequality for the spherical distance and the  fact that $(a+b)^{2H} \leq a^{2H} + b^{2H}$ for non-negative $a,b$ and $2H \leq 1$ (and we use $H\leq 1/2$).
\begin{lem}\label{posOfCovar}
    For all $\eta,\zeta\in\sph_{d-1}$,
    \begin{align*}
        \expec{S_H(\eta) S_H(\zeta)} = \eh\left[ d(\eta,O)^{2H} + d(\zeta,O)^{2H} - d(\eta, \zeta)^{2H} \right] \geq 0.
    \end{align*}
\end{lem}

Finally, let us combine stochastic and geometric estimates to obtain the following central lemma. To this end, we recall the notation for antipodal points: If $\eta \in\sph_{d-1}$ then $\ol{\eta}$ denotes the point antipodal to $\eta$. In particular, $\hh(\eta)\cup \hh(\ol{\eta}) = \sph_{d-1}$ always holds.

\begin{lem}\label{lem_combinationTopSlep}
    Let $(S_H(\eta))_{\eta\in\sph_{d-1}}$ be SFBM and let $(B_H(t))_{t\in\R^{d-1}}$ be  FBM on $\R^{d-1}$. Let $\kappa > 0$ be a constant and $M\subseteq \sph_{d-1}$. Then
    \begin{align*}
        \prob{S_H(\eta) \leq \kappa \quad \forall \eta\in M } 
        \quad \geq \quad c ~ \prob{B_H(\Psi(\eta)) \leq \kappa \quad \forall \eta \in M \cap \hh(O) },
    \end{align*}
    where $\Psi$ is the projection from Definition~\ref{def_aeprojection} and where $c > 0$ is a constant depending on $H$ and $d$. If $M\subseteq \hh(O)$ then the inequality simplifies to
    \begin{align*}
        \prob{S_H(\eta) \leq \kappa \quad \forall \eta\in M } 
        \quad \geq \quad \prob{B_H(\Psi(\eta)) \leq \kappa \quad \forall \eta \in M}.
    \end{align*}
\end{lem}
\begin{proof}
   Lemma~\ref{lem_aepProperties} \ref{lem_aepProperties_isom} allows us to infer that,  for all $\eta \in\hh(O)$,
   \begin{align*}
       \expec{S_H(\eta)^2} = d(\eta,O)^{2H} = \norm{\Psi(\eta)}_2^{2H} = \expec{B_H(\Psi(\eta))^2}, 
   \end{align*}
   and, using Lemma~\ref{lem_aepProperties} \ref{lem_aepProperties_top} we obtain that for all $\eta,\zeta\in\hh(O)$
    \begin{align*}
        \expec{S_H(\eta) S_H(\zeta)} 
        & = \eh \left( d(\eta, O)^{2H} + d(\zeta,O)^{2H} - d(\eta,\zeta)^{2H} \right)& \\
        & \geq \eh \left( \norm{\Psi(\eta)}_2^{2H} + \norm{\Psi(\zeta)}_2^{2H} - \norm{\Psi(\eta)-\Psi(\zeta)}_2^{2H} \right) & \\
        & \geq \expec{B_H(\Psi(\eta)) B_H(\Psi(\zeta))}.
    \end{align*}   
    
    Therefore, we can first apply Corollary~\ref{cor_SlepianPosCorr} and after that Lemma~\ref{lem_Slepian} to infer that
    \begin{align*}
        & \prob{S_H(\eta) \leq \kappa \; \forall \eta\in M} \\
        & \geq \prob{S_H(\eta) \leq \kappa \; \forall \eta\in M \cap \hh(\ol{O})} \cdot \prob{S_H(\eta) \leq \kappa \; \forall \eta\in  M \cap \hh(O)} \\
        & \geq \prob{S_H(\eta) \leq 0 \; \forall \eta\in \hh(\ol{O})} \cdot \prob{B_H(\Psi(\eta)) \leq \kappa \; \forall \eta\in M \cap \hh(O)}.
    \end{align*}
    In the case $M\subseteq \hh(O)$, the first factor after the first inequality is simply equal to $1$, which yields the second statement in the lemma.

The first statement also follows from the last display noticing that the constant $c:=\prob{S_H(\eta) \leq 0 \; \forall \eta\in \hh(\ol{O})}$ is strictly positive (cf.\ Lemma~\ref{lem:constantpositive}).
\end{proof}

\subsection{Lower bound}\label{sec_lowerBound}

Given all of the preparatory work, the proof of the lower bound is now rather easy.

\begin{lem}\label{lem_mainLowerBound}
    Let $(S_H(\eta))_{\eta\in \sph_{d-1}}$ be spherical fractional Brownian motion on $\sph_{d-1}$. Then
    \begin{align*}
        \prob{\sup_{\eta\in \sph_{d-1}} S_H(\eta) \leq \eps} \geq c ~ \eps^{\frac{d-1}{H}} \left|\log(\eps) \right|^{-\frac{d-1}{2H}}
    \end{align*}
    holds for some constant $c >0$ depending on $H$ and $d$, and $\eps>0$ small enough.
\end{lem}
\begin{proof}
    We first apply Lemma~\ref{lem_combinationTopSlep} with $M=\sph_{d-1}$ and $\kappa=\eps$ to obtain
    \begin{align*}
         \prob{\sup_{\eta\in \sph_{d-1}} S_H(\eta) \leq \eps} 
        & \geq c ~ \prob{\sup_{\eta\in \hh(O)} B_H(\Psi(\eta)) \leq \eps} \\
        & = c ~ \prob{\sup_{ t\in\R^{d-1}: \norm{t}_2 < \frac{\pi}{2} } B_H(t) \leq \eps} \\
        & = c ~ \prob{\sup_{ t\in\R^{d-1}: \norm{t}_2 < \frac{\pi}{2}\eps^{-\frac{1}{H}} } B_H(t) \leq 1},
    \end{align*}
    where the last equality uses the self-similarity property of FBM. The last probability can be estimated further by Lemma~4 in~\cite{molchan99}: For a constant $c=c(d,H)$,
    \begin{align*}
        \prob{\sup_{ t\in\R^{d-1}: \norm{t}_2 < \frac{\pi}{2}\eps^{-\frac{1}{H}} } B_H(t) \leq 1}
        \geq c ~ \eps^{\frac{d-1}{H}} \left|\log(\eps) \right|^{-\frac{d-1}{2H}}.\qquad\qquad \qedhere
    \end{align*}
\end{proof}

\subsection{Upper bound}\label{sec_upperBound}
The proof of the upper bound is slightly more involved than the proof of the lower bound. We require two further ingredients. First, the following lemma relates the Gaussian measure of a shifted event with that of the unshifted event, where the shift is from the RKHS. The lemma was proven as Proposition~3.1 in~\cite{aurder13}. Here, we apply a slightly corrected formulation, which can be found as Proposition~3 in~\cite{buck2017}.
\begin{lem}\label{lemRKHSBound}
    Let $X$ be some centred Gaussian process with RKHS denoted by $\mathbb{H}$. Let $\norm{.}$ be the norm in $\mathbb{H}$. Then for each $f\in\mathbb{H}$ and each measurable $A$ such that $\prob{X\in A}\in(0,1)$, we have
    \begin{align}\label{eq_RKHSBound_lower}
        e^{-\sqrt{2\norm{f}^2 \log( 1/ \prob{X\in A} )}  - \norm{f}^2 / 2} \prob{X\in A} \leq \prob{X+f\in A}.
    \end{align}
    If $\norm{f}^2 < 2\log( 1/ \prob{X\in A} ) )$ we additionally have
    \begin{align*} 
        \prob{X+f\in A} \leq e^{\sqrt{2\norm{f}^2 \log( 1/ \prob{X\in A} )}  - \norm{f}^2/2} \prob{X\in A}.
    \end{align*}
\end{lem}

Furthermore, we require a recent result on the occupation time of SFBM below zero given as Theorem 2 in \cite{aurdorpit2024}. 
\begin{lem} 
\label{lem_occTimeSFBM}
    Let $Z_-$ be the occupation time of SFBM below zero defined by
    \begin{align*}
        Z_-:= \frac{1}{\abs{\sph_{d-1}}} \int_{\sph_{d-1}} \indi{S_H(\xi) \leq 0} \dd\sigma(\xi).
    \end{align*}
    Then $Z_-$ is uniformly distributed on $[0,1]$.
\end{lem}

The original result is phrased for occupation time above zero, but symmetry yields the the the same for $Z_-$. We are now ready to prove the upper bound in our main result.

\begin{lem} 
    Let $(S_H(\eta))_{\eta\in \sph_{d-1}}$ be spherical fractional Brownian motion on $\sph_{d-1}$. Then
    \begin{align*}
         \prob{\sup_{\eta\in \sph_{d-1}} S_H(\eta) \leq \eps} 
        \leq \eps^{\frac{d-1}{H} + o(1)}, && \text{ for } \eps \searrow 0.
    \end{align*}
\end{lem}

\begin{proof}
    Fix an $\alpha > 0$, consider the function $\fda$ from Lemma~\ref{lemRKHS_SFBM_fdelta}, and set $\vphi := 1-\fda$ (we omit the subindices on $\vphi$). We define the parameter $0 < \delta := \eps^{1/(H+\alpha)} < 1$.

    We apply Inequality~(\ref{eq_RKHSBound_lower}) of Lemma~\ref{lemRKHSBound} with the event
    \begin{align*}
        A := A_{\eps} := \{ (x_\eta)_{\eta\in\sph_{d-1}} : x_\eta \leq \eps ~ \forall \eta\in\sph_{d-1} \}.
    \end{align*}
    and obtain
    \begin{align}
        & \prob{\sup_{\eta\in \sph_{d-1}} S_H(\eta) \leq \eps}
        = \prob{(S_H(\eta))_{\eta\in\sph_{d-1}} \in A} \notag \\
        & \quad \leq e^{\sqrt{2\norm{\eps\, \fda}^2 \log( 1/ \prob{(S_H(\eta))_{\eta\in\sph_{d-1}} \in A}}  + \norm{\eps\, \fda}^2/2} \prob{ S_H(\eta) + \eps \fda(\eta) \leq \eps ~ \forall \eta\in\sph_{d-1}}. \label{eqDereichApplied}
    \end{align}
    
    We first prove that, for all $\eps<1$,
    \begin{align}\label{eq_ShInA_upperBound}
        \prob{ S_H(\eta) + \eps\fda(\eta) \leq \eps ~ \forall \eta\in\sph_{d-1}} 
        \leq \frac{2}{\vol{d-1}} \eps^{\frac{d-1}{H+\alpha}}.
    \end{align}
    Indeed, using the definition of $\vphi$ in the first step, the fact that $\fda\geq 1$ on $\sph_{d-1} \setminus \bb_\delta(O)$ by Lemma~\ref{lemRKHS_SFBM_fdelta} \ref{RKHS_SFBM_fda_d3_prop_incr} (and thus $\vphi\leq 0$ on $\sph_{d-1} \setminus \bb_\delta(O)$) in the second step, the definition of $Z_-$ in the third step, and Lemma~\ref{lem_occTimeSFBM} in the fourth step, we obtain
    \begin{align*}
        \prob{ S_H(\eta) + \eps \fda(\eta) \leq \eps ~ \forall \eta\in\sph_{d-1}}
        & = \prob{ S_H(\eta) < \eps~ \vphi(\eta) \;\forall \eta\in \sph_{d-1}} \\
        & \leq \prob{ \sup_{\eta\in \sph_{d-1} \setminus \bb_\delta(O)} S_H(\eta) < 0} \\
        & \leq \prob{ Z_- \;>\; 1- \frac{\abs{\bb_\delta(O)}}{\vol{d-1}} }\\
        & = \frac{\abs{\bb_\delta(O)}}{\vol{d-1}}
        \leq \frac{2 \delta^{d-1}}{\vol{d-1}} 
        = \frac{2}{\vol{d-1}} \eps^{\frac{d-1}{H+\alpha}}.
    \end{align*}
    The upper bound on the volume $\abs{\bb_\delta(O)}$ can be obtained using Lemma~\ref{lem_simpleRadialFunctionIntegration} and the non-restrictive assumption $\delta \leq 1$.
    
    Next, we know using Lemma~\ref{lem_mainLowerBound} that    
    \begin{align}\label{eq_ShInA_lowerBound}
        \prob{ (S_H(\eta))_{\eta\in\sph_{d-1}} \in A}  \geq c ~ \eps^{\frac{d-1}{H}} \left|\log(\eps) \right|^{-\frac{d-1}{2H}} \geq \eps^{\frac{2(d-1)}{H}},
    \end{align}
    for any $\eps < \eps_0$, where $\eps_0$ is some constant depending on $H$ and $d$.

    The RKHS norm $\norm{\eps\, \fda}^2 \leq c$ is uniformly bounded in $\delta = \eps^{1/(H+\alpha)}$ using Lemma~\ref{lemRKHS_SFBM_fdelta} \ref{RKHS_SFBM_fda_d3_prop_normbound} for some constant $c$, which depends on $\alpha, H$ and $d$. Therefore, we can continue from Equation (\ref{eqDereichApplied}) and obtain using (\ref{eq_ShInA_upperBound}) and (\ref{eq_ShInA_lowerBound}) that
    \begin{align*}
       &  e^{\sqrt{2\norm{\eps\, \fda}^2 \log\left(( 1/ \prob{(S_H(\eta))_{\eta\in\sph_{d-1}} \in A}\right)}  + \norm{\eps\, \fda}^2/2} \prob{ S_H(\eta) + \eps\fda(\eta) \leq \eps ~ \forall \eta\in\sph_{d-1}} \\
        & \leq e^{\sqrt{2c^2 \left(\frac{2(d-1)}{H}\right)\abs{\log \eps}} + c^2/2} \frac{2}{\vol{d-1}} \eps^{\frac{d-1}{H+\alpha}}
    \end{align*}
    as soon as $\eps < \eps_0$. Thus we may infer that
    \begin{align*}
        \lim_{\eps\searrow 0} \frac{1}{\log(\eps)} \log  \prob{\sup_{\eta\in \sph_{d-1}} S_H(\eta) \leq \eps}
        \geq \frac{d-1}{H+\alpha}
    \end{align*}
    for any $\alpha > 0$. The statement follows by letting $\alpha \searrow 0$.
\end{proof}

{\bf Acknowledgements.} We would like to thank Moritz Egert and Karsten Gro\ss e-Brauckmann for discussions on different aspects of this paper.
\normalem

\bibliographystyle{alpha}
\bibliography{bibliography.bib}

\appendix

\section{Appendix}

The appendix collects some technical proofs that are slightly off the main thread of the paper.

\begin{proof}[Proof of Lemma~\ref{lem_powerSeries}]
    Let $\eps > 0$ and define the partial series of $p$ and $q$ to be
    \begin{align*}
        p_n(x) := \sum_{i=0}^n a_i x^i, \qquad
        q_n(x) := \sum_{j=0}^n b_j (x-u)^{\gamma+j}
    \end{align*}
    for any $n\in\N$. Note that $(q_n)$ is uniformly convergent, because it is the product of a function bounded on intervals with a uniformly convergent power series. In particular, uniform continuity of $(x-u)^\gamma$ implies uniform continuity of $q$. Using the assumption that $\widehat{q}([u,u+s))\subseteq (-r,r)$ we can infer that because of 
    \begin{align*}
        \abs{\sum_{j=0}^\infty b_j (x-u)^{\gamma+j} } \leq \sum_{j=0}^\infty \abs{b_j} (x-u)^{\gamma+j} < r,
    \end{align*}
    the range $q([u,u+s))\subseteq (-r,r)$. Since $q$ is additionally continuous, and $q(u) = 0$, we may choose any $0<s'<s$ such that $q([u,u+s'])\subseteq [-r_1, r_1]$ for some $0 < r_1 < r$. Now, choose any $r_2$ with $r_1 < r_2 < r$. Using uniform convergence of $(q_n)$ and uniform continuity of $p$, there is an $M_1\in\N$, s.t.\ for all $M\geq M_1$ we have
    \begin{align*}
        q_M(x) \in [-r_2, r_2] & \text{ for all $x\in[u,u+s']$,} \\
        \abs{p(q(x)) - p(q_M(x))}  < \eps & \text{ for all $x\in[u,u+s']$.}
    \end{align*}
    Using uniform convergence of $(p_n)$, there is an $N_1\in\N$, s.t.\ for all $N\geq N_1$ we have
    \begin{align*}
        \abs{p(y) - p_N(y)} < \eps & \text{ for all $y\in [-r_2,r_2]$.}
    \end{align*}
    Thus we can now infer that for any $N\geq N_1$ and $M\geq M_1$ and all $x\in[u,u+s']$, we have
    \begin{align*}
        \abs{p(q(x)) - p_N(q_M(x))} & \leq \abs{p(q(x)) - p(q_M(x))} + \abs{p(q_M(x)) - p_N(q_M(x))} < 2~\eps,
    \end{align*}
    i.e.\ the sequence $p_N(q_M(x))$ is uniformly convergent on $[u,u+s']$, if $N$ and $M$ are simultaneously sent to infinity. Since we assumed that $\widehat{q}([u,u+s))\subseteq (-r,r)$ and absolute convergence of both series, the argument is identical in the case that we replace the coefficients of both series by their absolute values. Thus, together with $\gamma > 0$ we may reorder the representation to obtain the absolutely convergent series
    \begin{align*}
        p(q(x)) = \sum_{i=0}^\infty d_i (x-u)^{\gamma_i}
    \end{align*}
    with a strictly increasing sequence $(\gamma_i)$ with $\gamma_0>0$.

    We want to show that the sequence of term-wise derivatives converges to the derivative of $p(q(x))$ for any $x\in (u,u+s')$ by showing that the sequence of term-wise derivatives converges uniformly on compact subsets $K \subseteq (u,u+s')$. Since $x\in K$ implies that $x-u$ is bounded away from $0$ and $s'$, the term
    \begin{align}\label{eq_supDerivativeSeries}
        \sup_{x\in K,~ y \geq 0} ~ y ~ \left(\frac{(x-u)}{s'}\right)^{y-1} \leq c
    \end{align}
    is bounded for some constant $c > 0$ depending on $K$. Since $\gamma_i > 0$, we can infer that
    \begin{align*}
        \abs{d_i} \gamma_i (x-u)^{\gamma_i-1} = \frac{\abs{d_i} (s')^{\gamma_i}}{s'} ~\gamma_i ~ \left(\frac{(x-u)}{s'}\right)^{\gamma_i-1} \leq \frac{c}{s'} \abs{d_i} (s')^{\gamma_i}.
    \end{align*}
    On the LHS we have the absolute value of the coefficients of the term-wise derivative of the series representation of $p(q(x))$ at some $x\in K$ and on the RHS we have a constant times the absolute value of the coefficients of the series representation of $p(q(x))$ at $x=s'$, which we know converges absolutely at $s'$. Thus, the sequence of term-wise derivatives is uniformly majorised by a convergent sequence and thus converges uniformly on $K$ and its limit is therefore the derivative of $p(q(x))$ for any $x\in K$. Thus, the sequence of term-wise derivatives converges pointwise for any $x\in (u,u+s')$.

    For higher derivatives, note that the function
    \begin{align*}
        y \mapsto & \sup_{x\in K} ~ y (y-1)\cdot \ldots \cdot (y-k) \left( \frac{x-u}{s'} \right)^{y-k-1}
    \end{align*}
    is bounded on $\R_{\geq 0}$. With this bound instead of (\ref{eq_supDerivativeSeries}), we obtain the same result for the $k$-th term-wise derivative, analogously.

    Finally, since the only restriction on $s'$ was $s'<s$, we obtain that the series representation converges uniformly on compacts, in particular pointwise, on the interval $[u,u+s)$, and that the sequence of term-wise $k$-th derivatives converges uniformly on compacts, in particular pointwise, on $(u,u+s)$. The proof for $\widetilde{q}$ can be carried out completely analogously.
\end{proof}

\begin{proof}[Proof of Lemma~\ref{lem_arccosRepAtSing}]
    We first examine the behaviour of Arccos for $t\searrow -1$ and make use of the fact that
    \begin{align*}
        \pi - \arccos(t) = \int_{-1}^t (1-x^2)^{-\eh} \dd x.
    \end{align*}
    Let $m\geq 0$ and note that
    \begin{align*}
        (1-x)^{-\eh} &= \sum_{i=0}^m (-1)^i \; 2^{-\eh-i} \binom{-\eh}{i} (1+x)^i + (-1)^{m+1} \; \binom{-\eh}{m+1} (1-a_x)^{\eh-m-2} (1+x)^{m+1} \\
        & =: T_m + R_{m}
    \end{align*}
    is the $m$-th order Taylor expansion $T_m$ of $(1-x)^{-\eh}$ at $x=-1$ with remainder $R_m$, where $-1<a_x<x$. Observe that
    \begin{align*}
        \int_{-1}^t (1+x)^{-\eh} T_m \dd x = \sum_{i=0}^m (-1)^i \; \frac{2^{-\eh-i}}{\eh+i} \binom{-\eh}{i} (1+t)^{\eh+i}.
    \end{align*}
    Let us subtract this integral from $\pi - \arccos(t)$ so that the remainder is given by
    \begin{align}\label{eq_fhAsympTaylorCalc1}
        \pi-\arccos(t) - \int_{-1}^t (1+x)^{-\eh} T_m \dd x 
        & = \int_{-1}^t (1-x^2)^{-\eh} - (1+x)^{-\eh} T_m \dd x \notag \\
        & = \int_{-1}^t (1+x)^{-\eh}\left( (1-x)^{-\eh} - T_m \right) \dd x \notag \\
        & = \int_{-1}^t (1+x)^{-\eh} R_{m} \dd x.
    \end{align}
    Since we are investigating the case $t\searrow -1$, we may assume that $t\leq 0$. This, together with the fact that $a_x < x \leq t \leq 0$, implies $(1-a_x)^{\eh-m-2} \leq 1$. Additionally using that $(-1)^{m+1} \binom{-\eh}{m+1}$ is non-negative for any $m\geq 0$, we obtain
    \begin{align*}
        R_{m} \leq (-1)^{m+1} \; \binom{-\eh}{m+1} (1+x)^{m+1}.
    \end{align*}
    Continuing from (\ref{eq_fhAsympTaylorCalc1}), we conclude that
    \begin{align*}
        \int_{-1}^t  (1+x)^{-\eh} R_{m} \dd x 
         & \leq (-1)^{m+1} \; \binom{-\eh}{m+1} \int_{-1}^t (1+x)^{\eh+m} \dd x \\
        & \in \OO\left((1+t)^{\eh+m+1}\right),
    \end{align*}
    which proves the desired expansion for $t\searrow -1$. The expansion for $t\nearrow +1$ is implied using the fact that $\arccos(-t) = \pi - \arccos(t)$.

    We need to check Assumption~\ref{assumptionLegendreAsymp}\ref{ass_term-wiseDiff}. To this end, note that
    \begin{align*}
        (-1)^i\binom{-\eh}{i} =
        \frac{\Gamma(i-\eh+1)}{\Gamma(i+1) \Gamma(1-\eh)} \sim c \; i^{-\eh},
    \end{align*}
    which allows us to conclude that
    \begin{align*}
        \sum_{i=1}^\infty \abs{(-1)^i \; \frac{2^{-\eh-i}}{\eh+i} \binom{-\eh}{i} (1+t)^{\eh+i}}
        & \leq c ~ \sum_{i=1}^\infty ~\frac{1}{\eh+i} ~ i^{-\eh} \left(\frac{1+t}{2}\right)^{\eh+i} \\
        & \leq c ~ \sum_{i=0}^\infty i^{-\eh-1} \left(\frac{1+t}{2}\right)^{\eh+i}
    \end{align*}
    converges absolutely on the domain $[-1,1]$. We can therefore split the series expansion at the singularities into a a power series and a remaining factor
    \begin{align*}
        \sum_{i=0}^\infty (-1)^i \; \frac{2^{-\eh-i}}{\eh+i} \binom{\eh-1}{i} (1+t)^{\eh+i}
        = (1+t)^\eh \sum_{i=0}^\infty (-1)^i \; \frac{2^{-\eh-i}}{\eh+i} \binom{-\eh}{i} (1+t)^{i}.
    \end{align*}
    The power series is absolutely convergent and therefore term-wise differentiable in the interior of its domain of convergence $[-3,1]$. The remaining factor is smooth on $(-1,\infty)$. Thus, we can infer that the expansion at the singularity is term-wise differentiable on $(-1,1)$. The same can be inferred by symmetry for the expansion at the singularity $+1$. 
\end{proof}

\begin{lem} Let $(S_H(\eta))_{\eta\in\mathbb{S}_{d-1}}$ be SFBM. 
    Then the constant 
    $$
    \prob{S_H(\eta) \leq 0 \; \forall \eta\in \hh(\ol{O})} 
    $$
    is strictly positive.
    \label{lem:constantpositive}
\end{lem}

\begin{proof}
    We will employ a version of the Kolmogorov-Chentsov continuity theorem. For a general reference in the Euclidean $d$-dimensional case, see \cite[Theorem 21.6 and Remark 21.7]{Klenke2020}. The case for more generalised metric spaces can be found in \cite[Theorem 1.1]{Kraetschmer2022}. The necessary assumptions for our SFBM may be inferred from Remark 1 and Proposition 3.1 (ii) in \cite{Kraetschmer2022}. We obtain that for any $h<H$ there exists an $h$-H\"older continuous modification $(S_H^*)$. W.l.o.g.\ we assume $(S_H^*) = (S_H)$. We furthermore obtain that there exists a non-random constant $K>0$, such that
    \begin{align*}
        \expec{\sup_{\substack{ \eta,\xi\in\sph_{d-1} \\ \eta\neq\xi }}\frac{\abs{S_H(\eta) - S_H(\xi)}}{d(\eta,\xi)^h} } \leq K.
    \end{align*}
    Defining $K_\eps:= \frac{K}{\eps}$, we may infer using Markov's inequality that
    \begin{align*}
        \prob{\exists \eta,\xi\in\sph_{d-1} : \abs{S_H(\eta) - S_H(\xi)} > K_\eps ~ d(\eta,\xi)^h } \leq \frac{\expec{\sup_{\substack{ \eta,\xi\in\sph_{d-1} \\ \eta\neq\xi }}\frac{\abs{S_H(\eta) - S_H(\xi)}}{d(\eta,\xi)^h} }}{K_\eps} \leq \eps.
    \end{align*}
    Define the event $C_{h,\eps} := \{\abs{S_H(\eta) - S_H(\xi)} \leq K_\eps ~ d(\eta,\xi)^h \quad \forall \eta,\xi\in\sph_{d-1}\}$. Let $\zeta\in\hh(\ol{O})$ and consider the spherical-distance ball $\bb_{K_\eps^{-1/h}}(\zeta)$ of radius $K_\eps^{-1/h}$ with centre $\zeta$ (cf.\ Eq. (\ref{def_sphDistBall})). We first prove that in the event $C_{h,\eps}$ it suffices to control the value of $S_H(\zeta)$ to infer non-positivity of $S_H(\eta)$ for all $\eta\in \bb_{K_\eps^{-1/h}}(\zeta)$, i.e.
    \begin{align*}
        \prob{\left\{ S_H(\zeta) < -1 \right\} \cap C_{h,\eps} }
        & = \prob{\left\{ S_H(\zeta) < -1 \right\} \cap \{\abs{S_H(\eta) - S_H(\xi)} \leq K_\eps ~ d(\eta,\xi)^h ~ \forall \eta,\xi\in\sph_{d-1}\} } \\
        & \leq \prob{\left\{ S_H(\zeta) < -1 \right\} \cap \{\abs{S_H(\eta) - S_H(\zeta)} \leq 1 ~ \forall \eta\in\bb_{K_\eps^{-1/h}(\zeta)\} } } \\
        & \leq \prob{S_H(\eta) \leq 0 \; \forall \eta\in \bb_{K_\eps^{-1/h}(\zeta)} }.
    \end{align*}
    Now let $X \sim \mathcal{N}\left(0,\left(\frac{\pi}{2}\right)^{2H}\right)$ be an independently chosen Gaussian random variable and let $\zeta \in \hh(\ol{O})$. Then
    \begin{align*}
        0 & < \prob{X < -1 }\\
        & \leq \prob{S_H(\zeta) < -1 }\\
        & = \prob{\left\{ S_H(\zeta) < -1 \right\} \cap C_{h,\eps} } + \prob{ \left\{S_H(\zeta) < -1\right\} \cap C_{h,\eps}^c } \\
        & \leq \prob{S_H(\eta) \leq 0 \; \forall \eta\in \bb_{K_\eps^{-1/h}(\zeta)} } + \prob{ C_{h,\eps}^c } \\
        & = \prob{S_H(\eta) \leq 0 \; \forall \eta\in \bb_{K_\eps^{-1/h}(\zeta)} } + \eps.
    \end{align*}
    Let us set $\eps:= \frac{\prob{X < -1 }}{2}$. Then the probability $\prob{S_H(\eta) \leq 0 \; \forall \eta\in \bb_{K_\eps^{-1/h}(\zeta)} }$ is uniformly bounded away from zero for all $\zeta\in\hh(\ol{O})$ by $\prob{X < -1 }/2$.  We may then choose finitely many $\zeta_1, \ldots, \zeta_N \in \hh(\ol{O})$ for some $N\in\N$ depending on $\eps$, such that the lower hemisphere can be covered by spherical-distance balls of radius $K_\eps^{-1/h}$ centred at $(\zeta_i)$, i.e.\ $\hh(\ol{O}) \subseteq \bigcup_{i=1,\ldots,N} \bb_{K_\eps^{-1/h}}(\zeta_i)$. Finally, using Slepian's lemma (in the version of Corollary \ref{cor_SlepianPosCorr}), we obtain
    \begin{align*}
        \prob{S_H(\eta) \leq 0 \; \forall \eta\in \hh(\ol{O})}
        & \geq \prod_{i=1}^N \prob{S_H(\eta) \leq 0 \; \forall \eta\in \bb_{K_\eps^{-1/h}}(\zeta_i)} \\
        & \geq \left[\frac{\prob{X < -1 }}{2}\right]^N > 0.\qquad\qquad \qedhere
    \end{align*}
\end{proof}

\end{document}